\documentclass{amsart}
\usepackage{amsmath}
\usepackage{amsthm}
\usepackage{amssymb}

\usepackage{diagrams}
\input xy
\xyoption{all}

 \newtheorem{thm}{Theorem}[section]
 \newtheorem{cor}[thm]{Corollary}
 \newtheorem{lem}[thm]{Lemma}
 \newtheorem{prop}[thm]{Proposition}
 \theoremstyle{definition}
 \newtheorem{defn}[thm]{Definition}
 \theoremstyle{remark}
 \newtheorem{rem}[thm]{Remark}
 \theoremstyle{remark}
 \newtheorem{ex}[thm]{Example}
 \numberwithin{equation}{section}

\newcommand{\pn}{\noindent}

\newcommand{\ZZ}{\mathbb{Z}}

\newcommand{\RR}{\mathbb{R}}
\newcommand{\LL}{\mathbb{L}}

\newcommand{\PP}{\mathbb{P}}

\newcommand{\Hom}{\mathrm{Hom}}
\newcommand{\HOM}{\mathrm{HOM}}

\newcommand{\Ext}{\mathrm{Ext}}
\newcommand{\Biext}{\mathrm{Biext}}
\newcommand{\cExt}{\mathcal{E}xt}
\newcommand{\cBiext}{\mathcal{B}iext}
\newcommand{\cPicard}{\mathcal{P}icard}
\newcommand{\cTorsor}{\mathcal{T}orsor}

\newcommand{\uAut}{\underline{\mathrm{Aut}}}
\newcommand{\bBiext}{\mathrm{\mathbf{Biext}}}
\newcommand{\bPicard}{\mathrm{Picard}}

\newcommand{\Add}{\mathrm{\mathbf{Add}}}

\newcommand{\coker}{\mathrm{coker}}
\newcommand{\im}{\mathrm{im}}

\newcommand{\h}{\mathrm{H}}
\newcommand{\R}{\mathrm{R}}
\newcommand{\E}{\mathrm{E}}
\newcommand{\rL}{\mathrm{L}}

\newcommand{\Tot}{\mathrm{Tot}}

\newcommand{\cD}{\mathcal{D}}

\newcommand{\cH}{\mathcal{H}}
\newcommand{\cK}{\mathcal{K}}

\newcommand{\bS}{\textbf{S}}

\newcommand{\eic}{\mathcal{E}}

\newcommand{\pic}{\mathcal{P}}
\newcommand{\qic}{\mathcal{Q}}
\newcommand{\ric}{\mathcal{R}}
\newcommand{\gic}{\mathcal{G}}
\newcommand{\bic}{\mathcal{B}}
\newcommand{\dic}{\mathcal{D}}
\newcommand{\lic}{\mathcal{L}}
\newcommand{\kic}{\mathcal{K}}

\begin{document}

\title[biextensions and Picard stacks]
{Biextensions of Picard stacks\\ and\\ their homological interpretation}

\author{Cristiana Bertolin}

\address{Dip. di Matematica, Universit\`a di Torino, Via Carlo Alberto 10, 
I-10123 Torino}
\email{cristiana.bertolin@googlemail.com}






\begin{abstract} 
Let {\bS} be a site. We introduce the 2-category of biextensions of strictly commutative Picard $\bS$-stacks. We define the pull-back, the push-down, and the sum of such biextensions and we compute their homological interpretation: if $\pic, \qic$ and $\gic$ are strictly commutative Picard $\bS$-stacks, the equivalence classes of biextensions of $(\pic,\qic)$ by $\gic$ are parametrized by the cohomology group $\Ext^1([\pic] {\otimes}^{\LL}[\qic] ,[\gic])$, the isomorphism classes of arrows from such a biextension to itself are parametrized by the cohomology group $\Ext^0([\pic]{\otimes}^{\LL} [\qic] ,[\gic])$ and the automorphisms of an arrow from such a biextension to itself are parametrized by the cohomology group $\Ext^{-1}([\pic]{\otimes}^{\LL}[\qic] ,[\gic])$,
  where $[\pic],[\qic]$ and $[\gic]$ are the complex associated to $\pic, \qic$ and $\gic$ respectively.
\end{abstract}


\maketitle


\tableofcontents

\section*{Introduction}

Let $\bS$ be a site. Let $P$, $Q$ and $G$ be three abelian sheaves on $\bS$.
In~\cite{SGA7} Expos\'e VII Corollary 3.6.5 Grothendieck proves that the group ${\Biext}^0(P,Q;G)$ of automorphisms of any biextension
of $(P,Q)$ by $G$ and the group ${\Biext}^1(P,Q;G)$ of isomorphism
 classes of biextensions of $(P,Q)$ by $G$, have the following homological interpretation:
\begin{equation}\label{intro:homological}
 {\Biext}^i(P,Q;G) \cong {\Ext}^i(P {\buildrel {\scriptscriptstyle \LL}
 \over \otimes} Q,G) \qquad (i=0,1)
\end{equation}
where $P {\buildrel{\scriptscriptstyle \LL} \over \otimes}Q$ is the derived functor of the functor $Q \rightarrow P \otimes Q$ in the derived category $\cD(\bS)$ of complexes of abelian sheaves on $\bS$. The aim of this paper is to find an analogous homological interpretation for biextensions of strictly commutative Picard $\bS$-stacks.

 Let $\pic, \qic$ and $\gic$ be three strictly commutative Picard $\bS$-stacks. We define a
 biextension of $(\pic,\qic)$ by $\gic$ as a $\gic_{\pic \times_{\mathbf{1}} \qic}$-torsor $\bic$ over $\pic \times_{\mathbf{1}} \qic$, endowed with a structure of extension of $\qic_\pic$
by $\gic_\pic$ and a structure of extension of $\pic_\qic$ by $\gic_\qic,$ which are compatible one with another. Biextensions of $(\pic,\qic)$ by $\gic$ form a 2-category ${\cBiext}(\pic,\qic;\gic)$ where
\begin{itemize}
	\item the objects are biextensions of $(\pic,\qic)$ by $\gic$,
	\item the 1-arrows are additive functors between biextensions,
	\item the 2-arrows are morphisms of additive functors.
\end{itemize}
Equivalence classes of biextensions of strictly commutative Picard $\bS$-stacks are endowed with a group law. We denote by ${\cBiext}^1(\pic,\qic;\gic)$ the group of equivalence classes of objects of ${\cBiext}(\pic,\qic;\gic)$, by ${\cBiext}^{0}(\pic,\qic;\gic)$ the group of isomorphism classes of arrows from an object of ${\cBiext}(\pic,\qic;\gic)$ to itself, and 
by ${\cBiext}^{-1}(\pic,\qic;\gic) $ the group of automorphisms of an arrow from an object of ${\cBiext}(\pic,\qic;\gic)$ to itself. With these notation our main Theorem is 

\begin{thm}\label{intro:mainthm}
Let $\pic,\qic$ and $\gic$ be strictly commutative Picard $\bS$-stacks. Then we have the following isomorphisms of groups 

	(a) ${\cBiext}^1(\pic,\qic;\gic) \cong 
{\Ext}^{1}\big([\pic]{\buildrel {\scriptscriptstyle \LL} \over \otimes}[\qic],[\gic]\big)=
{\Hom}_{\cD(\bS)}\big([\pic]{\buildrel {\scriptscriptstyle \LL}
 \over \otimes}[\qic],[\gic][1]\big),$
		
	(b) ${\cBiext}^0(\pic,\qic;\gic) \cong 
{\Ext}^{0}\big([\pic]{\buildrel {\scriptscriptstyle \LL}\over \otimes}[\qic],[\gic]\big)= 
{\Hom}_{\cD(\bS)}\big([\pic]{\buildrel {\scriptscriptstyle \LL} \over \otimes}[\qic],[\gic]\big),$

	(c) ${\cBiext}^{-1}(\pic,\qic;\gic) \cong  
{\Ext}^{-1}\big([\pic]{\buildrel {\scriptscriptstyle \LL} \over \otimes}[\qic],[\gic]\big)=
{\Hom}_{\cD(\bS)}\big([\pic]{\buildrel {\scriptscriptstyle \LL}
 \over \otimes}[\qic],[\gic][-1]\big),$ \\
where $[\pic],[\qic]$ and $[\gic]$ denote the complex of $\cD^{[-1,0]}(\bS)$ corresponding to $\pic,\qic$ and $\gic$ respectively. 
\end{thm}

By \cite{SGA4} \S 1.4 there is an equivalence of categories between the category of strictly commutative Picard $\bS$-stacks and the derived category $\cD^{[-1,0]}(\bS)$ of complexes $K$ of abelian sheaves on $\bS$ such that ${\h}^i (K)=0$ for $i \not= -1$ or $0$. 
 Via this equivalence, the above notion of biextension of strictly commutative Picard $\bS$-stacks furnishes a notion of biextension for complexes of abelian sheaves over $\bS$ concentrated in degrees -1 and 0 and the above theorem  generalizes Grothendieck's result (\ref{intro:homological}) to complexes of abelian sheaves concentrated in degrees -1 and 0.
 
The definitions and results of this paper generalizes those of \cite{Be09}: in fact, in loc.cit. we have defined the notion of biextensions of 1-motives and we 
have checked Theorem \ref{intro:mainthm} for 1-motives (recall that a 1-motive can be seen as a complex of abelian sheaves $[u: A \rightarrow B] \in \cD^{[-1,0]}(\bS)$).

The main Theorem \ref{intro:mainthm} furnishes also the homological interpretation of extensions of strictly commutative Picard $\bS$-stacks which was computed in \cite{Be10}:
in fact, it $\mathbf{1}$ is the strictly commutative Picard $\bS$-stack 
such that for any object $U$ of $\bS$,
$\mathbf{1}(U)$ is the category with one object and one arrow, then
\begin{itemize}
	\item the 2-category ${\cBiext}(\pic,\mathbf{1};\gic)$ of biextensions of $(\pic,\mathbf{1})$ by $\gic$
is equivalent to the 2-category ${\cExt}(\pic,\gic)$ of extensions of $\pic$ by $\gic$, and
	\item in the derived category 
$ {\Ext}^i([\pic] {\buildrel {\scriptscriptstyle \LL} \over \otimes} [\mathbf{1}],[\gic]) \cong {\Ext}^i([\pic],[\gic]) $ for $i=-1,0,1$.
\end{itemize}

In~\cite{SGA7} Expos\'e VII Grothendieck states the following geometrical-homological principle: \emph{if an abelian sheaf $A$ on $\bS$ admits an explicit representation
 in $\cD(\bS)$ by a complex ${\rL}.$
whose components are direct sums of objects of the kind ${\ZZ}[I]$,
with $I$ a sheaf on $\bS$, then the groups ${\Ext}^i(A,B)$ admit
an explicit geometrical description for any abelian sheaf $B$ on $\bS$.}\\
A first example of this principle is furnished by
the geometrical notion of extension of abelian sheaves on $\bS$: in fact, if $P$ and
$G$ are two abelian sheaves on $\bS$, it is a classical result that the group
${\Ext}^0(P,G)$ is isomorphic to the group of automorphisms of any extension of $P$ by $G$ and the group ${\Ext}^1(P,G)$ is isomorphic to the group of isomorphism classes of extensions of $P$ by $G.$ The canonical isomorphisms (\ref{intro:homological}) are another  of this Grothendeick's principle 
which involves the geometrical notion of biextension of abelian sheaves.
Other examples of this Grothendeick's principle are described in~\cite{B83}:
If $P$ and $G$ are abelian sheaves on $\bS$,
according to loc.cit. Proposition 8.4, the strictly commutative Picard $\bS$-stack of symmetric biextensions of $(P,P)$ by $G$ is equivalent to the strictly commutative Picard $\bS$-stack associated to the object
 $\tau_{\leq 0} {\RR}{\Hom}({\LL} \mathrm{Sym}^2 (P) ,G[1])$ of $\cD(\bS)$, and according to loc.cit. Theorem 8.9, the strictly commutative Picard $\bS$-stack of the 3-tuple $(L,E,\alpha)$ (resp. the 4-tuple $(L,E,\alpha,\beta)$) defining a cubic structure (resp. a $\Sigma$-structure) on the $G$-torsor $L$ is equivalent to the strictly commutative Picard $\bS$-stack associated to the object $\tau_{\leq 0} {\RR}{\Hom}({\LL} P^+_2 (P),G[1])$
(resp. $\tau_{\leq 0} {\RR}{\Hom}({\LL} \Gamma_2 (P),G[1])$) of $\cD(\bS)$.
Our Theorem~\ref{intro:mainthm} is the first example in the literature where the geometrical-homological principle of Grothendieck is applied to complexes of abelian sheaves.

A strictly commutative Picard $\bS$-2-stack is the 2-analog of a
strictly commutative
Picard $\bS$-stack, i.e. it is an $\bS$-2-stack in 2-groupoids $\PP$ endowed with a morphism of $\bS$-2-stacks $ +: {\PP} \times_{\bS} {\PP} \rightarrow {\PP}$ and with
associative and commutative constraints (see \cite{T} Definition 2.3 for more details). As for strictly commutative Picard $\bS$-stacks and complexes of abelian sheaves concentrated in degrees -1 and 0, in \cite{T} Tatar proves that there is a dictionary between strictly commutative Picard $\bS$-2-stacks and complexes of abelian sheaves concentrated in degrees -2, -1 and 0. 
Using this dictionary, we can rewrite Theorem~\ref{intro:mainthm} as followed: the strictly commutative Picard $\bS$-2-stack of biextensions of $(\pic,\gic)$ by $\qic$ is equivalent to the strictly commutative Picard $\bS$-2-stack associated to the object 
$$\tau_{\leq 0} {\R}{\Hom}\big([\pic]{\buildrel {\scriptscriptstyle \LL} \over \otimes}[\qic],[\gic][1]\big)$$
of $\cD^{[-2,0]}(\bS)$.
If $\mathbf{1}$ denotes the strictly commutative Picard $\bS$-stack such that for any object $U$ of $\bS$,
$\mathbf{1}(U)$ is the category with one object and one arrow, biextensions of $(\pic,\mathbf{1})$ by $\gic$ are just extensions of $\pic$ by $\gic$. According to this remark, Theorem~\ref{intro:mainthm} furnishes another proof of
the main result of \cite{Be10} which states that the strictly commutative Picard $\bS$-2-stack of extensions of $\pic$ by $\gic$ is equivalent to the strictly commutative Picard $\bS$-2-stack associated to the object 
$$\tau_{\leq 0} {\R}{\Hom}\big([\pic],[\gic][1]\big)$$
of $\cD^{[-2,0]}(\bS)$.

This paper is organized as followed: in Section 1 we recall some basic results on the 2-category of strictly commutative Picard $\bS$-stacks. Let $\gic$ be a gr-$\bS$-stack. In Section 2 we define the notions of $\gic$-torsor, morphism of $\gic$-torsors and morphism of morphisms of $\gic$-torsors, getting the 2-category of $\gic$-torsors. In Section 3 we recall some basic results on the 2-category of extensions of strictly commutative Picard $\bS$-stacks.
Let $\pic$ and $\gic$ be two strictly commutative Picard $\bS$-stacks.
In Section 4 we prove that it exists an equivalence of 2-categories between the 2-category of extensions of $\pic$ by $\gic$ and the 2-category consisting of the data $(\eic,I,M,\alpha,\chi)$, where $\eic$ is a $\gic$-torsor over $\pic$, $I$ is 
a trivialization of its pull-back via the additive functor $\mathbf{1}: \mathbf{e} \rightarrow \pic$, $M : p_1^* \; \eic \wedge p_2^* \; \eic \longrightarrow +^* \; \eic$ is a morphism of $\gic$-torsors (where $+$ is the group law of $\pic$ and $p_i: \pic \times \pic \rightarrow \pic$ are the projections), and $\alpha$ and $\chi$ are two isomorphisms of morphisms of $\gic$-torsors involving the morphism of $\gic$-torsors $M$ (Theorem \ref{thm:ext-tor}). This generalizes to strictly commutative Picard $\bS$-stacks the following result of Grothendieck (\cite{SGA7} Expos\'e VII 1.1.6 and 1.2): if $P$ and $Q$ are two abelian sheaves, to have an extension of $P$ by $G$ is the same thing as to have the 4-tuple
$(P,G,E,\varphi)$, where $E$ is a $G_P$-torsor over $P$, and $\varphi: pr_1^*E ~ pr_2^*E \rightarrow +^* E$ is an isomorphism of torsors over $P\times P$ satisfying some associativity and commutativity conditions. Let $\pic, \qic,\gic$ be three strictly commutative Picard $\bS$-stacks.
In Section 5 we define the notions of biextension of $(\pic, \qic)$ by $\gic$, morphism of such biextensions and morphism of morphisms of such biextensions, getting the 2-category of biextensions of $(\pic, \qic)$ by $\gic$. In Section 6 
we introduce the notions of pull-back and push-down of biextensions of strictly commutative Picard $\bS$-stacks. This will allow us to define a
group law for equivalence classes of biextensions of strictly commutative Picard $\bS$-stacks. In Section 7 we prove the cases $(b)$ and $(c)$ of Theorem \ref{intro:mainthm}. In order to prove the case $(a)$ we need to introduce an intermediate 2-category $\Psi_{\lic^.}(\gic)$ that we construct using a strictly commutative Picard $\bS$-stack $\gic$ and a complex $\lic^.$ of strictly commutative Picard $\bS$-stacks (Section 8). This 2-category $\Psi_{\lic^.}(\gic)$ has the following \emph{homological description}:
\begin{equation}\label{intro:psi-homo}
    \Psi_{\lic^.}^i(\gic) \cong {\Ext}^i({\Tot}([\lic^.]),[\gic]) \qquad (i=0,1,2)
\end{equation}
where $\Psi_{\lic^.}^1(\gic)$ is the group of equivalence classes of objects of 
$\Psi_{\lic^.}(\gic)$, $\Psi_{\lic^.}^0(\gic)$ is the group of isomorphism classes of arrows from an object of $\Psi_{\lic^.}(\gic)$ to itself, and $\Psi_{\lic^.}^{-1}(\gic)$ is the group of automorphisms of an arrow from an object of $\Psi_{\lic^.}(\gic)$ to itself.
In section 9, to any complex of the kind $[\pic]=[d^P: P^{-1} \rightarrow P^0]$ we associate a canonical flat partial resolution $[\lic^.(\pic)]$ 
whose components are direct sums of objects of the kind ${\ZZ}[I]$ with $I$ an abelian sheaf on $\bS$. Here ``partial resolution'' means that we have an isomorphism between the cohomology groups of $[\pic]$ and of this partial resolution only in degree 1, 0 and -1. This is enough
for our goal since only the groups $\Ext^1, \Ext^0$ and $\Ext^{-1}$ are involved in the statement of Theorem~\ref{intro:mainthm}. The category $\Psi_{\lic^.(\pic) \otimes \lic^.(\gic)}(\gic)$ admit the following \emph{geometrical description}:
 \begin{equation}\label{intro:psi-geo}
 \Psi_{\lic^.(\pic) \otimes \lic^.(\gic)}(\gic)  \cong   {\cBiext}([\pic],[\qic];[\gic])
\end{equation}
Putting together this geometrical description~(\ref{intro:psi-geo}) with the homological description~(\ref{intro:psi-homo}), in Section 10 we finally prove Theorem~\ref{intro:mainthm}.

\section*{Notation}

Let $\bS$ be a site. 
Denote by $\cK(\bS)$ the category of complexes of abelian sheaves on the site $\bS$: all complexes that we consider in this paper are cochain complexes (excepted in Section 9 and 10 where we switch to homological notation).
Let $\cK^{[-1,0]}(\bS)$ be the subcategory of $\cK(\bS)$ consisting of complexes $K=(K^i)_i$ such that $K^i=0$ for $i \not= -1$ or $0$. The good truncation $ \tau_{\leq n} K$ of a complex $K$ of $\cK(\bS)$ is the following complex: $ (\tau_{\leq n} K)^i= K^i$ for $i <n,  (\tau_{\leq n} K)^n= \ker(d^n)$
and $ (\tau_{\leq n} K)^i= 0$ for $i > n.$ For any $i \in {\ZZ}$, the shift functor $[i]:\cK(\bS) \rightarrow \cK(\bS) $ acts on a  complex $K=(K^n)_n$ as $(K[i])^n=K^{i+n}$ and $d^n_{K[i]}=(-1)^{i} d^{n+i}_{K}$.
If ${\rL}^{..}$ is a bicomplex of abelian sheaves on the site $\bS$, we denote by
${\Tot}({\rL}^{..})$ the total complex of ${\rL}^{..}$: it is the cochain complex whose component of degree $n$ is ${\Tot}({\rL}^{..})^n=\sum_{i+j=n} {\rL}^{ij}.$

Denote by $\cD(\bS)$ the derived category of the category of abelian sheaves on $\bS$, and let $\cD^{[-1,0]}(\bS)$ be the subcategory of $\cD(\bS)$ consisting of complexes $K$ such that ${\h}^i (K)=0$ for $i \not= -1$ or $0$. If $K$ and $K'$ are complexes of $\cD(\bS)$, the group ${\Ext}^i(K,K')$ is by definition ${\Hom}_{\cD(\bS)}(K,K'[i])$ for any $i \in {\ZZ}$. Let ${\R}{\Hom}(-,-)$ be the derived functor of the bifunctor ${\Hom}(-,-)$. The cohomology groups\\ ${\h}^i\big({\R}{\Hom}(K,K') \big)$ of 
${\R}{\Hom}(K,K')$ are isomorphic to ${\Hom}_{\cD(\bS)}(K,K'[i])$.

A \textbf{2-category} $\mathcal{A}=(A,C(a,b),K_{a,b,c},U_{a})_{a,b,c \in A}$ is given by the following data:
\begin{itemize}
  \item a set $A$ of objects $a,b,c, ...$;
  \item for each ordered pair $(a,b)$ of objects of $A$, a category $C(a,b)$;
  \item for each ordered triple $(a,b,c)$ of objects $A$, a functor
      $K_{a,b,c}:C(b,c) \times C(a,b) \longrightarrow C(a,c),$
      called composition functor. This composition functor have to satisfy the associativity law;
  \item for each object $a$, a functor $U_a:\mathbf{1} \rightarrow C(a,a)$ where \textbf{1} is the terminal category (i.e. the category with one object, one arrow), called unit functor. This unit functor have to provide a left and right identity for the composition functor.
\end{itemize}

This set of axioms for a 2-category is exactly like the set of axioms for a category in which the arrows-sets ${\Hom}(a,b)$ have been replaced by the categories $C(a,b)$.
We call the categories $C(a,b)$ (with $a,b \in A$) the \textbf{categories of morphisms} of the 2-category $\mathcal{A}$: the objects of $C(a,b)$ are the \textbf{1-arrows} of $\mathcal{A}$ and the arrows of $C(a,b)$ are the \textbf{2-arrows} of $\mathcal{A}$.

Let $\mathcal{A}=(A,C(a,b),K_{a,b,c},U_{a})_{a,b,c \in A} $ and $ \mathcal{A}'=(A',C(a',b'),K_{a',b',c'},U_{a'})_{a',b',c' \in A'}$ be two 2-categories. A \textbf{2-functor} (called also a \textbf{morphism of 2-categories})
$$(F,F_{a,b})_{a,b \in A}: \mathcal{A} \longrightarrow \mathcal{A}'$$
consists of
\begin{itemize}
  \item an application $F: A \rightarrow A'$ between the objects of $\mathcal{A}$ and the objects of $\mathcal{A}'$,
  \item a family of functors $F_{a,b}:C(a,b) \rightarrow C(F(a),F(b))$ (with $a,b \in A$) which are compatible with the composition functors and with the unit functors of $\mathcal{A}$ and $\mathcal{A}'$.
\end{itemize}

\section{The 2-category of Picard stacks}

Let {\bS} be a site. For the notions of {\bS}-pre-stack, {\bS}-stack, morphism of {\bS}-stacks and morphism of morphisms of {\bS}-stacks we refer to \cite{G} Chapter II 1.2.

A \textbf{strictly commutative Picard {\bS}-stack} consists of an {\bS}-stack of groupoids $\pic$, a morphism of {\bS}-stacks $ +: \pic \times \pic \rightarrow \pic$ (called the group law of $\pic$) with two natural isomorphisms of associativity $\sigma$ and of commutativity $\tau$, which are described by the functorial isomorphisms
\begin{eqnarray}
\label{ass}  \sigma_{a,b,c} &:& (a + b) + c \stackrel{\cong}{\longrightarrow} a + ( b + c) \qquad \forall~ a,b,c \in \pic, \\
\label{com} \tau_{a,b} &:& a + b \stackrel{\cong}{\longrightarrow} b + a \qquad \forall~ a,b \in \pic,
\end{eqnarray}
 a global neutral object $e$ with two natural isomorphisms 
\begin{equation}\label{eds}
 l_a: e+ a \stackrel{\cong}{\longrightarrow} a, \qquad  
 r_a: a+e \stackrel{\cong}{\longrightarrow} a \qquad \forall~ a \in \pic
\end{equation}
which coincide on $e$ ($l_e=r_e$), and finally a morphism of {\bS}-stacks $ -: \pic \rightarrow \pic$  with two natural isomorphisms 
\begin{equation}\label{inv}
 o_a: a+(-a) \stackrel{\cong}{\longrightarrow} e, \qquad  
c_{ab}: -(a+b) \stackrel{\cong}{\longrightarrow} (-a) +(-b) \qquad \forall~ a \in \pic.
\end{equation}
These data have to satisfy the following conditions: 
\begin{itemize}
	\item the natural isomorphism $\sigma$ is coherent, i.e. for any $a,b,c$ and $d \in \pic$ the following pentagonal diagram commute 
\begin{equation}\label{pic-1}
\xymatrix{
   a+(b +(c+d)) &  \ar[l]_{\sigma} (a+b) +(c+d)   &((a+b) +c)+d \ar[d]^{\sigma +id_\pic} \ar[l]_\sigma \\
a+((b +c)+d)  \ar[u]^{id_\pic + \sigma} & &(a+(b +c))+d \ar[ll]_\sigma  
}
\end{equation}
	\item for any $a \in \pic,$
\begin{equation}\label{pic-2}
	 \tau_{a,a}: a + a \longrightarrow a + a
\end{equation}
is the identity; This condition, which is seldom verified, justifies the terminology \emph{strictly} commutative.
	\item the natural isomorphism $\tau$ is coherent, i.e. for any $a$ and $b\in \pic$ the following diagram commute 
\begin{equation}\label{pic-3}
\xymatrix{
   a+b \ar[r]^\tau \ar[dr]_{id_\pic} & b+a \ar[d]^\tau \\
 & a+b
}
\end{equation}	
	\item the natural isomorphisms $\sigma$ and $\tau$ are compatible, i.e. for any $a,b$ and $c \in \pic$ the following hexagonal diagram commute 
\begin{equation}\label{pic-4}
 \xymatrix{
   &  b+(c +a)   & \\
(b+c)+a  \ar[ur]^\sigma & &b+(a+c) \ar[ul]_{id_\pic + \tau } \\
a+(b+c) \ar[u]^\tau & & (b+a)+c \ar[u]_\sigma \\
  &  (a+b) +c \ar[ur]_{\tau + id_\pic} \ar[ul]^\sigma  &
}
\end{equation}
\item the natural isomorphism $\sigma$ and the neutral object are compatible, i.e. for any $a $ and $b\in \pic$ the following diagram commute 
\begin{equation}\label{pic-5}
\xymatrix{
   (a+e)+b \ar[r]^\sigma \ar[dr]_{r+id_\pic} & a+(e+b) \ar[d]^{id_\pic +l} \\
 & a+b.
}
\end{equation}
\end{itemize}
In particular, for any object $U$ of $\bS$, $(\pic (U),+,e,-)$ is a strictly commutative Picard category (see Definition 1.4.2 \cite{SGA4}). 
The sheaf of automorphisms of the neutral object ${\uAut}(e)$ is abelian.
If $\pic$ and $\qic$ are two strictly commutative Picard {\bS}-stacks,
 an \textbf{additive functor} $(F,\sum):\pic \rightarrow \qic $
 is a morphism of {\bS}-stacks $F:\pic \rightarrow \qic $ endowed with a natural isomorphism $\sum$ which is described by the functorial isomorphisms
\[\sum_{a,b}:F(a+b) \stackrel{\cong}{\longrightarrow} F(a)+F(b)
\qquad \forall~ a,b \in \pic \]
and which is compatible with the natural isomorphisms $\sigma$ and $\tau$ of
$\pic$ and $\qic$. A \textbf{morphism of additive functors $\alpha:(F,\sum) \Rightarrow (F',\sum') $} is a morphism of morphisms of {\bS}-stacks $\alpha:F \Rightarrow F' $ which is compatible with the natural isomorphisms $\sum$ and $\sum'$ of $F$ and $F'$ respectively.
We denote by ${\Add}_{\bS} (\pic,\qic)$
the category whose objects are additive functors from $\pic$ to $\qic$ and whose arrows are morphisms of additive functors. The category ${\Add}_{\bS} (\pic,\qic)$ is a groupoid, i.e. any morphism of additive functors is an isomorphism of additive functors.

An \textbf{equivalence of strictly commutative Picard {\bS}-stacks} between $\pic$ and $\qic$ is an additive functor $(F,\sum):\pic \rightarrow \qic$ with $F$ an equivalence of {\bS}-stacks. Two strictly commutative Picard {\bS}-stacks are \emph{equivalent as strictly commutative Picard {\bS}-stacks} if there exists an equivalence of strictly commutative Picard {\bS}-stacks between them.

To any strictly commutative Picard {\bS}-stack $\pic$, we associate the sheaffification $\pi_0(\pic)$ of the pre-sheaf which associates to each object $U$ of $\bS$ the group of isomorphism classes of objects of $\pic(U)$,
 the sheaf $\pi_1(\pic)$ of automorphisms ${\uAut}(e)$ of the neutral object of $\pic$, and an element $\varepsilon(\pic)$ of ${\Ext}^2(\pi_0(\pic),\pi_1(\pic))$. 
Two strictly commutative Picard {\bS}-stacks $\pic$ and $\pic'$ are equivalent as strictly commutative Picard {\bS}-stacks if and only if $\pi_i(\pic)$ is isomorphic to $\pi_i(\pic')$ for $i=0,1$ and $\varepsilon(\pic)=\varepsilon(\pic')$.

A \textbf{strictly commutative Picard {\bS}-pre-stack} consists of 
an {\bS}-pre-stack of groupoids $\pic$, a morphism of {\bS}-stacks $ +: \pic \times \pic \rightarrow \pic$ with a natural isomorphism of associativity $\sigma$ (\ref{ass}),
a global neutral object $e$ with two natural isomorphisms $r$ and $l$ (\ref{eds}),
and a morphism of {\bS}-stacks $ -: \pic \rightarrow \pic$  with two natural isomorphisms $o$ and $c$ (\ref{inv}), such that
 for any object $U$ of $\bS$, $(\pic (U),+,e,-)$ is a strictly commutative Picard category. If $\pic$ is a strictly commutative Picard \bS-pre-stack, there exists modulo a unique equivalence
one and only one pair $(a \pic, j)$ where $a\pic$ is a strictly commutative Picard \bS-stack and $j:  \pic \rightarrow a \pic$ is an additive functor. $(a \pic, j)$ is \emph{the strictly commutative Picard \bS-stack generated by $\pic$.}

In~\cite{SGA4} \S 1.4 Deligne associates to each complex
$K$
of $\cK^{[-1,0]}(\bS)$ a strictly commutative Picard {\bS}-stack $st(K)$ and to each morphism of complexes $g: K \rightarrow L$ an additive functor 
$st(g):st(K) \rightarrow st(L)$
between the strictly commutative Picard {\bS}-stacks associated to the complexes $K$ and $L$. 
Moreover he proves the following links between strictly commutative Picard {\bS}-stacks and complexes of $\cK^{[-1,0]}(\bS)$, between additive functors and morphisms of complexes and between morphisms of additive functors and homotopies of complexes:
\begin{itemize}
	\item for any strictly commutative Picard $\bS$-stack $\pic$ there exists a complex $K$ of $\cK^{[-1,0]}(\bS)$ such that $\pic = st(K)$;
	\item if $K,L$ are two complexes of $\cK^{[-1,0]}(\bS)$, then for any additive functor $F: st(K) \rightarrow st(L)$ there exists a quasi-isomorphism $k:K' \rightarrow K$ and a morphism of complexes $l:K' \rightarrow L$ such that $F$ is isomorphic as additive functor to $st(l) \circ st(k)^{-1}$;
	\item if $f,g: K \rightarrow L$ are two morphisms of complexes of $\cK^{[-1,0]}(\bS)$, then
\begin{equation}
{\Hom}_{{\Add}_{\bS}(st(K),st(L))}(st(f),st(g)) \cong \Big\{ \mathrm{homotopies}~ H:K \rightarrow L ~~|~~ g-f=dH+Hd \Big\}.
\label{eq:homotopies}
\end{equation}
\end{itemize}
Denote by ${\bPicard}(\bS)$ the category whose objects are small strictly commutative Picard $\bS$-stacks and whose arrows are isomorphism classes of additive functors. The above links between strictly commutative Picard $\bS$-stacks and complexes of abelian sheaves on $\bS$ furnish the equivalence of category:
\begin{eqnarray}
\label{st} st: \cD^{[-1,0]}(\bS) &\longrightarrow & {\bPicard}(\bS) \\
 \nonumber  K & \mapsto & st(K) \\
 \nonumber K \stackrel{f}{\rightarrow} L & \mapsto &  st(K) \stackrel{st(f)}{\rightarrow} st(L).
\end{eqnarray}
We denote by $[\,\,]$ the inverse equivalence of $st$. 
Let ${\cPicard}(\bS)$ be the 2-category of strictly commutative Picard $\bS$-stacks whose objects are strictly commutative Picard $\bS$-stacks and whose categories of morphisms are the categories ${\Add}_{\bS}(\pic,\qic)$ (i.e.
       the 1-arrows are additive functors between strictly commutative Picard $\bS$-stacks and the 2-arrows are morphisms of additive functors).
Via the functor $st$, there exists a 2-functor between

	(a) the 2-category whose objects and 1-arrows are the objects and the arrows of the category $\cK^{[-1,0]}(\bS)$ and whose 2-arrows are the homotopies between 1-arrows (i.e. $H$ such that $g-f=dH+Hd$ with $f,g: K \rightarrow L$ 1-arrows),
 
 	(b) the 2-category ${\cPicard}(\bS)$.

\begin{ex}\label{ex:pic} Let $\pic, \qic$ and $\gic$ be three strictly commutative Picard {\bS}-stacks. \\
I) Let 
$${\HOM}(\pic,\qic)$$
be the strictly commutative Picard {\bS}-stack defined as followed:
 for any object $U$ of $\bS$, the objects of the category ${\HOM}(\pic,\qic)(U)$ are additive functors from ${\pic}_{|U}$ to ${\qic}_{|U}$ and its arrows are morphisms of additive functors. 
By (\ref{eq:homotopies}) and (\ref{st}), we have the equality 
$[{\HOM}(\pic,\qic)] = \tau_{\leq 0}{\R}{\Hom}\big([\pic],[\qic]\big)$ in the derived category $\cD^{[-1,0]}(\bS).$\\
II) A \textbf{biadditive functor} $(F,l,r):\pic \times \qic \rightarrow \gic $
 is a morphism of {\bS}-stacks $F:\pic \times \qic \rightarrow \gic $ endowed with two natural isomorphisms, which are described by the functorial isomorphisms
\begin{eqnarray}
 \nonumber l_{a,b,c}:F(a+b,c) &\stackrel{\cong}{\longrightarrow}& F(a,c)+F(b,c) \qquad \forall~ a,b \in \pic, \; \forall~ c \in \qic  \\
 \nonumber r_{a,c,d}:F(a,c+d) &\stackrel{\cong}{\longrightarrow}& F(a,c)+F(a,d)
 \qquad \forall~ a \in \pic, \; \forall~ c,d \in \qic,
\end{eqnarray}
 such that 
\begin{itemize}
	\item for any fixed $a \in \pic$, $F(a,-)$ is compatible with the natural isomorphisms $\sigma$ and $\tau$ of $\pic$ and $\gic$,
       \item for any fixed $c \in \qic$, $F(-,c)$ is compatible with the natural isomorphisms $\sigma$ and $\tau$ of $\qic$ and $\gic$,
        \item for any fixed $a,b\in \pic$ and $c,d \in \qic$ is the following diagram commute 
  \[    \xymatrix{
     F(a+b,c+d)\ar[r]^r  \ar[d]_l& F(a+b,c)+F(a+b,d) \ar[r]^{l+l \qquad}& F(a,c)+F(b,c)+F(a,d)+F(b,d)   \\
   F(a,c+d)+F(b,c+d) \ar[rr]_{r+r}  &  & F(a,c)+F(a,d)+F(b,c)+F(b,d) \ar[u]_{id_\gic +\tau+ id_\gic} .
}\]  
\end{itemize}
A \textbf{morphism of biadditive functors $\alpha:(F,l,r) \Rightarrow (F',l',r') $} is a morphism of morphisms of {\bS}-stacks $\alpha:F \Rightarrow F' $ which is compatible with the natural isomorphisms $l,r$ and $l',r$ of $F$ and $F'$ respectively.
 Let 
 $${\HOM}(\pic,\qic;\gic)$$
be the strictly commutative Picard {\bS}-stack defined as followed:
 for any object $U$ of $\bS$, the objects of the category ${\HOM}(\pic,\qic;\gic)(U)$ are biadditive functors from ${\pic}_{|U} \times {\qic}_{|U}$ to ${\gic}_{|U}$ and its arrows are morphisms of biadditive functors.\\
 III) Let 
 $$\pic \otimes \qic$$
  be the strictly commutative Picard $\bS$-stack endowed with a biadditive functor $\otimes :\pic \times \qic \rightarrow \pic \otimes \qic$ such that for any strictly commutative Picard $\bS$-stack $\gic$, the biadditive functor $\otimes$ defines the following equivalence of strictly commutative Picard $\bS$-stacks: 
\begin{equation}
{\HOM}(\pic \otimes \qic, \gic) \cong  {\HOM}(\pic,\qic;\gic).
\label{eq:tensorlinear}
\end{equation}
According to \cite{SGA4} 1.4.20, in the derived category $\cD^{[-1,0]}(\bS)$ we have the equality 
$[\pic\otimes \qic] = \tau_{\geq -1} ([\pic] \otimes^\LL[\qic]) $.
\end{ex}

According to \S 2 \cite{Be10} we have the following operations on strictly commutative Picard $\bS$-stacks:\\
(1) The \textbf{product} of two strictly commutative Picard $\bS$-stacks $\pic$ and $\qic$ is the strictly commutative Picard $\bS$-stack $\pic \times \qic$ defined as followed:
\begin{itemize}
  \item for any object $U$ of $\bS$, an object of the category $\pic \times \qic(U)$ is a pair $(p,q)$ of objects with $p$ an object of $\pic(U)$ and $q$ an object of $\qic(U)$;
  \item for any object $U$ of $\bS$, if $(p,q)$ and $(p',q')$ are two objects of 
  $\pic \times \qic(U)$, an arrow of $\pic \times \qic(U)$ from $(p,q)$ to $(p',q')$ is a pair $(f,g)$ of arrows with $f:p \rightarrow  p' $ an arrow of $\pic(U)$ and $g:q \rightarrow  q' $ an arrow of $\qic(U)$.
\end{itemize}
(2) Let $G:\pic \rightarrow \qic$ and $F:\pic' \rightarrow \qic$ be additive functors between strictly commutative Picard $\bS$-stacks. The \textbf{fibered product} of $\pic$ and $\pic'$ over $\qic$ via $F$ and $G$ is the strictly commutative Picard $\bS$-stack $\pic \times_\qic \pic'$ defined as followed: 
\begin{itemize}
  \item for any object $U$ of {\bS}, the objects of the category $(\pic \times_\qic \pic')(U)$ are triplets $(p,p',f)$ where $p$ is an object of $\pic(U)$, $p' $ is an object of $\pic'(U)$ and $f:G(p) \stackrel{\cong}{\rightarrow}F(p')$ is an isomorphism of $\qic(U)$ between $G(p)$ and $F(p')$;
  \item for any object $U$ of {\bS}, if $(p_1,p'_1,f)$ and $(p_2,p'_2,g)$ are two objects of
  $(\pic \times_\qic \pic')(U)$, an arrow of $(\pic \times_\qic \pic')(U)$ 
 from  $(p_1,p'_1,f)$ to $(p_2,p'_2,g)$ is a pair $(f,g)$ of arrows with
  $\alpha:p_1 \rightarrow p_2$ of arrow of $\pic(U)$ and $\beta:p'_1 \rightarrow p'_2$ an arrow of $\pic'(U)$ such that $ g \circ G(\alpha) = F(\beta) \circ f$.
\end{itemize}
The fibered product $\pic \times_\qic \pic'$ is also called the \textbf{pull-back} $F^*\pic$ of $\pic$ via $F:\pic' \rightarrow \qic$ or the \textbf{pull-back} $G^*\pic'$ of $\pic'$ via $G:\pic \rightarrow \qic$.\\
(3) Let $G:\qic \rightarrow \pic$ and $F:\qic \rightarrow \pic'$ be additive functors between strictly commutative Picard $\bS$-stacks. The \textbf{fibered sum} of $\pic$ and $\pic'$ under $\qic$ via $F$ and $G$ is the strictly commutative Picard $\bS$-stack $\pic +^\qic \pic'$ generated by the following strictly commutative Picard $\bS$-pre-stack $\dic$:
\begin{itemize}
  \item for any object $U$ of {\bS}, the objects of the category $\dic(U)$ are the objects of the category $(\pic \times \pic')(U)$, i.e. pairs $(p,p')$ 
with $p$ an object of $\pic(U)$ and  $p'$ an object of $\pic'(U)$;
  \item for any object $U$ of {\bS}, if $(p_1,p'_1)$ and $(p_2,p'_2)$ are two objects of $\dic(U)$, 
an arrow of $\dic (U)$ from $(p_1,p'_1)$ to $(p_2,p'_2)$ is an equivalence class of triplets   
  $(q,\alpha,\beta)$ with $q$ an object of $\qic (U)$, $\alpha: p_1 + G(q) \rightarrow p_2$ an arrow of $\pic(U)$ and $\beta: p'_1 + F(q) \rightarrow p'_2$ 
an arrow of $\pic'(U)$. Two triplets $(q_1,\alpha_1,\beta_1)$ and $(q_2,\alpha_2,\beta_2)$
are equivalent it there is an arrow $\gamma: q_1 \rightarrow q_2$ in $\qic (U)$ such that $\alpha_2 \circ (id +G(\gamma)) = \alpha_1$ and
 $ (F(\gamma) + id) \circ \beta_1 =\beta_2 $.  
\end{itemize}
The fibered sum $\pic +^\qic \pic'$ is also called the \textbf{push-down} $F_*\pic$ of $\pic$ via $F:\qic \rightarrow \pic'$ or the \textbf{push-down} $G_*\pic'$ of $\pic'$ via $G:\qic \rightarrow \pic$.

We have analogous operations on complexes of $\cK^{[-1,0]}(\bS)$:\\
(1) The \textbf{product} of two complexes $P=[d^P:P^{-1} \rightarrow P^0]$ and $ Q=[d^Q: Q^{-1} \rightarrow Q^0]$ of $\cK^{[-1,0]}(\bS)$ is the complex $P+Q=[(d^P,d^Q):P^{-1} + Q^{-1}\rightarrow P^0 +Q^0]$. Via the equivalence of category (\ref{st}) we have that $st(P+ Q) = st(P) \times st(Q)$.\\
(2) Let $P=[d^P:P^{-1} \rightarrow P^0], Q=[d^Q: Q^{-1} \rightarrow Q^0]$ and $G=[d^G: G^{-1} \rightarrow G^0]$ be complexes of $\cK^{[-1,0]}(\bS)$ and let $f:P \rightarrow G$ and $g:Q \rightarrow G$ be morphisms of complexes. The \textbf{fibered product} $P \times_G Q$ of $P$ and $Q$ over $G$
 is the complex $[d_P \times_{d_G} d_Q: P^{-1} \times_{G^{-1}} Q^{-1} \rightarrow P^{0} \times_{G^0} Q^{0}]$,
where for $i=-1,0$ the abelian sheaf $P^{i} \times_{G^i} Q^{i}$ is the fibered product of $P^i$ and of $Q^i$ over $G^i$ and the morphism of abelian sheaves $d_P \times_{d_G}d_Q $ is given by the universal property of the fibered product 
$P^{0} \times_{G^0} Q^{0}$. The fibered product $P \times_G Q$ is also called the \textbf{pull-back} $g^*P$ of $P$ via $g:Q \rightarrow G$ or the \textbf{pull-back} $f^*Q$ of $Q$ via $f:P \rightarrow G$.
Remark that $st(P \times_G Q) = st(P) \times_{st(G)} st(Q)$ via the equivalence of category (\ref{st}).\\
(3) Let $P=[d^P:P^{-1} \rightarrow P^0], Q=[d^Q: Q^{-1} \rightarrow Q^0]$ and $G=[d^G: G^{-1} \rightarrow G^0]$ be complexes of $\cK^{[-1,0]}(\bS)$ and let $f:G \rightarrow P$ and $g:G \rightarrow Q$ be morphisms of complexes. The \textbf{fibered sum} $P +^G Q$ of $P$ and $Q$ under $G$
 is the complex $[d_P +^{d_G} d_Q: P^{-1} +^{G^{-1}} Q^{-1} \rightarrow P^{0} +^{G^0} Q^{0}]$,
where for $i=-1,0$ the abelian sheaf $P^{i} +^{G^i} Q^{i}$ is the fibered sum of $P^i$ and of $Q^i$ under $G^i$ and the morphism of abelian sheaves $d_P +^{d_G}d_Q $ is given by the universal property of the fibered sum 
$P^{-1} +^{G^{-1}} Q^{-1}$. The fibered sum $P +^G Q$ is also called the \textbf{push-down} $g_*P$ of $P$ via $g:G \rightarrow Q$ or the \textbf{push-down} $f_*Q$ of $Q$ via $f:G \rightarrow P$.
We have $st(P +^G Q) = st(P) +^{st(G)} st(Q)$ via the equivalence of category (\ref{st}).

\section{The 2-category of $\gic$-torsors}

Let $\gic$ be a gr-$\bS$-stack, i.e. an {\bS}-stack of groupoids $\gic$ equipped with the following data: a morphism of {\bS}-stacks $ +: \gic \times \gic \rightarrow \gic$ with a natural isomorphism of associativity $\sigma$ (\ref{ass}),
a global neutral object $e$ with two natural isomorphisms $r$ and $l$ (\ref{eds}),
and a morphism of {\bS}-stacks $ -: \pic \rightarrow \pic$  with two natural isomorphisms $o$ and $c$ (\ref{inv}), such that
 for any object $U$ of $\bS$, $(\gic (U),+,e,-)$ is a gr-category (i.e. see \cite{B92} \S 3.1 for more details). Remark that a strictly commutative Picard {\bS}-stack is a gr-$\bS$-stack endowed with a strict commutative condition $\tau$ (\ref{com}) and (\ref{pic-2}).

\begin{defn}
A \textbf{left $\gic$-torsor} $\pic=(\pic, M, \mu)$ consists of 
\begin{itemize}
	\item an {\bS}-stack of groupoids $\pic$,
	\item a morphism of {\bS}-stacks $M: \gic \times \pic \rightarrow \pic$, and
	\item an isomorphism of morphisms of {\bS}-stacks $\mu: M \circ (+ \times id_\pic)  \Rightarrow  M \circ (id_\gic \times M) $ 
	\[\xymatrix{
 \gic \times \gic \times \pic \qquad  \ar[r]^{+ \, \times id_\pic}  \ar[d]_{id_\gic \times M}&  \gic \times \pic \ar[d]^M \ar@{=>}[dl]^\mu\\
  \gic \times \pic \ar[r]_M & \pic
}
\]
 which is described by the functorial isomorphism 
$\mu_{g_1,g_2,p}: M(g_1+g_2,p) \rightarrow M(g_1,M(g_2,p))$ for any $g_1,g_2 \in \gic$ and $p \in \pic$,
\end{itemize}
 such that the following conditions are satisfied:

	(i) the natural isomorphism $\mu$ is compatible with the natural isomorphism of associativity $\sigma$ underlying $\gic$, i.e. the following diagram commute for any $g_1,g_2,g_3 \in \gic$ and $p \in \pic$
 \[\xymatrix{
  M( (g_1+g_2) +g_3,p) \ar[rr]^{M(\sigma,id_\pic)} \ar[d]_\mu&  & M(g_1+(g_2 +g_3) ,p)  \ar[d]^\mu\\
M( g_1+g_2, M(g_3,p))\ar[dr]_\mu &  &M(g_1, M(g_2 +g_3 ,p))\ar[dl]^\mu\\
   &M(g_1,M(g_2,M(g_3,p))) & 
}
\]
   
     (ii) the restriction of the morphism of {\bS}-stacks $M$ to $\mathbf{e} \times \pic$ is equivalent to the identity, i.e. 
$M(e,p) \cong p$ for any $p \in \pic$ (here $\mathbf{e}$ denotes the gr-$\bS$-stack such that for any object $U$ of $\bS$,
$\mathbf{e}(U)$ is the category consisting of the neutral object $e$ of $\gic$). Moreover we require that this restriction of $M$ is compatible with the natural isomorphism $\mu$, i.e. the following diagrams commute for any $g \in \gic$ and $p \in \pic$  
  \[\xymatrix{
  M( g+e,p) \ar[rr]^\mu \ar[dr] & & M(g,M(e,p)) \ar[dl] \\
 &M(g,p) &
}
\]
 \[\xymatrix{
  M( e+g,p) \ar[rr]^\mu \ar[dr] & & M(e,M(g,p)) \ar[dl] \\
 &M(g,p) &
}
\]
	
	(iii) the morphism of {\bS}-stacks $(M,Pr_2): \gic \times \pic \rightarrow \pic \times \pic$ is an equivalence of {\bS}-stacks (here $Pr_2:\gic \times \pic \rightarrow \pic $ denotes the second projection),

	(iv) it exists a covering sieve $R$ of the site {\bS} such that for any object $U$ of $R$ the category $\pic(U)$ is not empty.
\end{defn}

\begin{defn}
A \textbf{morphism of left $\gic$-torsors} 
$$(F,\gamma):(\pic, M, \mu) \rightarrow (\pic',M',\mu')$$
 consists of 
\begin{itemize}
	\item a morphism of {\bS}-stacks $F: \pic \rightarrow \pic'$ and
	\item an isomorphism of morphisms of {\bS}-stacks $\gamma: M' \circ (id_\gic \times F)  \Rightarrow  F \circ M $ described by the functorial isomorphism 
$\gamma_{g,p}: M'(g,F(p)) \rightarrow F(M(g,p))$ for any $g \in \gic$ and $p \in \pic$, 
\end{itemize}
 which are compatible with the natural isomorphisms $\mu$ and $\mu'$, i.e. the following diagram commute for any $g_1,g_2 \in \gic$ and $p \in \pic$
 \[\xymatrix{
 M'(g_1+g_2,F(p)) \ar[d]_{\gamma_{g_1+g_2,p}}\quad \ar[r]^{\mu'_{g_1,g_2,F(p)}} &  \quad M'(g_1 ,M'(g_2,F(p))) \quad \ar[r]^{M'(id_\gic,\gamma_{g_2,p})}& \quad M'(g_1 ,F(M(g_2,p))) \ar[d]^{\gamma_{g_1,M(g_2,p)}}\\
F(M(g_1+g_2,p)) \ar[rr]_{F(\mu_{g_1,g_2,p})}& &F(M(g_1,M(g_2,p))).
}
\]
\end{defn}

Let $(F,\gamma),(\overline{F},\overline{\gamma}):(\pic, M, \mu) \rightarrow (\pic',M',\mu')$ be two morphisms of left $\gic$-torsors.

\begin{defn}
A \textbf{morphism of morphisms of left $\gic$-torsors} 
$$\varphi: (F,\gamma) \Rightarrow (\overline{F},\overline{\gamma})$$
consists of a morphism of morphisms of {\bS}-stacks $\varphi: F \Rightarrow \overline{F}$ which is compatible with the natural isomorphisms $\gamma$ and $\overline{\gamma}$, i.e.
the following diagram commute for any $g \in \gic$ and $p \in \pic$
 \[\xymatrix{
 M'(g,F(p)) \ar[d]_{M'(id_\gic,\varphi(p))}\quad \ar[r]^{\gamma} & F(M(g,p))\ar[d]^{\varphi(M(g,p))}\\
M'(g,\overline{F}(p)) \ar[r]_{\overline{\gamma}}& \overline{F}(M(g,p)).
}
\]
\end{defn}

If the gr-$\bS$-stack $\gic$ acts of the right side instead of the left side, we get the definitions of right $\gic$-torsor, morphism of right $\gic$-torsors and morphism of morphisms of right $\gic$-torsors.

\begin{defn}\label{def:tor}
 A \textbf{$\gic$-torsor} $\pic=(\pic, M_r,M_l,\mu_r,\mu_l,\kappa)$ consists of 
 an {\bS}-stack of groupoids $\pic$ endowed with a structure of left $\gic$-torsor $(\pic,M_l,\mu_l)$ and a structure of right $\gic$-torsor $(\pic,M_r,\mu_r)$ which are compatible one with another. This compatibility is given by the existence of 
an isomorphism of morphisms of {\bS}-stacks $\kappa: M_l \circ (id_\gic \times M_r)  \Rightarrow  M_r \circ (M_l \times id_\gic )$, described by the functorial isomorphism 
$\kappa_{g_1,p,g_2}: M_l(g_1,M_r(p,g_2)) \rightarrow M_r(M_l(g_1,p),g_2)$ for any $g_1,g_2 \in \gic$ and $p \in \pic$, such that the following diagrams commute for any $g_1,g_2 \in \gic$ and $p \in \pic$
\[\xymatrix{
  M_l( g_1+g_2,M_r(p,g_3)) \ar[rr]^{\kappa} \ar[d]_{\mu_l}&  & M_r(M_l(g_1+g_2,p),g_3)   \ar[d]^{M_r(\mu_l,id_\gic)}\\
M_l( g_1,M_l(g_2, M_r(p,g_3)))\ar[dr]_\kappa &  &M_r(M_l(g_1, M_l(g_2,p),g_3)\\
   &M_l(g_1,M_r(M_l(g_2,p),g_3)) \ar[ur]_\kappa& 
}
\]
\[\xymatrix{
  M_l( g_1,M_r(p,g_2+g_3)) \ar[rr]^{\kappa} \ar[d]_{M_l(id_\gic,\mu_r)}&  & M_r(M_l(g_1,p),g_2+g_3)   \ar[d]^{\mu_r}\\
M_l( g_1,M_r(M_r(p,g_2),g_3))\ar[dr]_\kappa &  &M_r(M_r(M_l(g_1,p),g_2),g_3)\\
   &M_r(M_l(g_1,M_r(p,g_2)),g_3) \ar[ur]_\kappa& 
}
\]
\end{defn}

\begin{ex} 
The strictly commutative Picard $\bS$-stack $\gic$ is endowed with a structure of $\gic$-torsor: the morphism of $\bS$-stacks $ +: \gic \times \gic \rightarrow \gic$ and the natural isomorphism of associativity $\sigma$ furnish a structure of left $\gic$-torsor and a structure of right $\gic$-torsor. The natural isomorphism of commutativity $\tau$ implies that these two structures are compatible, i.e. $\gic$ is in fact a $\gic$-torsor. We will call $\gic$ the trivial $\gic$-torsor.
\end{ex}

\begin{defn}\label{def:mortor}
A \textbf{morphism of $\gic$-torsors} 
$$(F,\gamma_r,\gamma_l):(\pic, M_r,M_l,\mu_r,\mu_l,\kappa) \rightarrow (\pic', M_r',M_l',\mu'_r,\mu'_l,\kappa')$$
 consists of 
\begin{itemize}
	\item a morphism of {\bS}-stacks $F: \pic \rightarrow \pic'$,
	\item two isomorphisms of morphisms of $\bS$-stacks
$(\gamma_l)_{g,p}: M'_l(g,F(p)) \rightarrow F(M_l(g,p))$ and 
$(\gamma_r)_{ p,g}: M'_r(F(p),g) \rightarrow F(M_r(p,g))$ for any $g \in \gic$ and $p \in \pic$,
\end{itemize}
such that $(F,\gamma_r):(\pic, M_r,\mu_r) \rightarrow (\pic', M_r',\mu'_r)$ and $(F,\gamma_l):(\pic, M_l,\mu_l) \rightarrow (\pic', M_l',\mu'_l)$ are morphisms of right respectively left $\gic$-torsors, and such that $\gamma_r$ and $\gamma_l$ are compatible with $\kappa$ and $\kappa'$, i.e. the following diagram commutate for any $g_1,g_2 \in \gic$ and $p \in \pic$
\[\xymatrix{
  M_l'(g_1,F(M_r(p,g_2))) \ar[r]^{\gamma_l}  & F(M_l(g_1,M_r(p,g_2))) \ar[r]^{F(\kappa)} & F(M_r(M_l(g_1,p),g_2))   \\
M_l'(g_1,M_r'(F(p),g_2))  \ar[u]^{M_l'(id_\gic,\gamma_r)} \ar[r]^{\kappa'} & M_r'(M_l'(g_1,F(p)),g_2)\quad \ar[r]^{M_r'(\gamma_l,id_\gic)} &\quad M_r'(F(M_l(g_1,p)),g_2)\ar[u]^{\gamma_r}.
}
\]
\end {defn}

Let $(F,\gamma_r,\gamma_l),(\overline{F},\overline{\gamma}_r, \overline{\gamma}_l):(\pic, M_r,M_l,\mu_r,\mu_l,\kappa) \rightarrow (\pic', M_r',M_l',\mu'_r,\mu'_l,\kappa')$ be two morphisms of $\gic$-torsors.

\begin{defn}\label{def:mormortor}
A \textbf{morphism of morphisms of $\gic$-torsors} 
$$\varphi: (F,\gamma_r,\gamma_l) \Rightarrow (\overline{F},\overline{\gamma}_r, \overline{\gamma}_l)$$
consists of a morphism of morphisms of {\bS}-stacks $\varphi: F \Rightarrow \overline{F}$ such that $\varphi: (F,\gamma_l) \Rightarrow (\overline{F}, \overline{\gamma}_l)$ and $\varphi: (F,\gamma_r) \Rightarrow (\overline{F},\overline{\gamma}_r)$ are morphisms of morphisms of left respectively right $\gic$-torsors, i.e. such that $\varphi: F \Rightarrow \overline{F}$ is compatible with the natural isomorphisms $\gamma_r, \overline{\gamma}_r$ and with the natural isomorphisms $\gamma_l, \overline{\gamma}_l.$
\end{defn}

$\gic$-torsors form a 2-category ${\cTorsor}(\gic)$
 where
\begin{enumerate}
	\item the objects are $\gic$-torsors,
	\item the 1-arrows are morphisms of $\gic$-torsors,
	\item the 2-arrows are morphisms of morphisms of $\gic$-torsors.
\end{enumerate}

Now we generalize to complexes of sheaves concentrated in degree -1 and 0, the classical notion of "torsor". 
Let $G=[d^G:G^{-1} \rightarrow G^0]$ be a complex of $\cK^{[-1,0]}(\bS)$, i.e. a complex of abelian sheaves on $\bS$ concentrated in degrees -1 and 0:
 
\begin{defn}
An left $G$-\textbf{torsor} $P=(P,m,\mu)$ consists of 
\begin{itemize}
	\item a complex $P=[d^P: P^{-1} \rightarrow P^0]$ of sheaves of sets on $\bS$ concentrated in degrees -1 and 0,
	\item a morphism of complexes $m: G \times P \rightarrow P$, i.e. a commutative diagram
	\[\xymatrix{
 G^{-1} \times P^{-1}  \ar[r]^{\qquad m^{-1}} \ar[d]_{d^G \times d^P}  &  P^{-1} \ar[d]^{d^P}\\
 G^0 \times P^0 \ar[r]_{\qquad m^0} & P^0
}
\]
	\item an homotopy $\mu$ between the two morphisms of complexes  $m \circ (+ \times id_P)$ and $m \circ (id_G \times m) $ from $G \times G\times P$ to $P$ (here $+: G \times G \rightarrow G$ is the group law underlying the complex of abelian sheaves $G$),
\end{itemize}
 such that the following conditions are satisfied:

	(i) the homotopy $\mu$ is compatible with the associative law of the complex of abelian sheaves $G$, i.e. the following diagram commute 
 \[\xymatrix{
  m \circ(+ \circ(+ \times id_G) \times id_P) \ar@{=}[rr] \ar[d]_\mu&  & m \circ(+ \circ(id_G \times +) \times id_P)  \ar[d]^\mu\\
m \circ (+\times m) \ar[dr]_\mu &  &m \circ (id_G \times m \circ (+ \times id_P))\ar[dl]^\mu\\
   &m \circ (id_G \times m) \circ (id_G \times id_G \times m) & 
}
\]

     (ii) the restriction of the morphism of complexes $m$ to $[id:e_{G^{-1}} \rightarrow e_{G^0}] \times \pic$ is homotopic to the identity (here $e_{G^{-1}}$ and $e_{G^0}$ denote the neutral sections of the abelian sheaves $G^{-1}$ and $G^0$ respectively). Moreover we require that this restriction of $m$ is compatible with the homotopy $\mu$, i.e. the following diagram commutes for any $g^i \in G^i$ and $p^i \in P^i$ for $i=-1,0$
  \[\xymatrix{
  m^i( g^i+e_{G^i},p^i) \ar[rr]^\mu \ar@{=}[dr] & & m^i(g^i,m^i(e_{G^i},p^i)) \ar[dl] \\
 &m^i(g^i,p^i) &
}
\]

	(iii) the morphism of complexes $(m,pr_2): G \times P \rightarrow P \times P$ is a quasi-isomorphism (here $pr_2:G \times P \rightarrow P $ denotes the second projection),
	
	(iv) it exists a covering sieve $R$ of the site {\bS} such that for any object $U$ of $R$ the sets of sections $P^{-1}(U)$ and $P^0(U)$ are not empty.
\end{defn}

\begin{defn}
A \textbf{morphism of left $G$-torsors} 
$$(f,\gamma):(P, m, \mu) \rightarrow (P',m',\mu')$$
 consists of 
\begin{itemize}
	\item a morphism of complexes $f: P \rightarrow P'$ and
	\item an homotopy $\gamma: m' \circ (id_G \times f) \approx f \circ m $, 
\end{itemize}
 which are compatible with the homotopies $\mu$ and $\mu'$, i.e. the following diagram commute 
 \[\xymatrix{
 m' \circ (+ \times f) \ar[d]_{\gamma}\quad \ar[r]^{\mu'} &  \quad m' \circ (id_G \times m' \circ (id_G \times f)) \quad \ar[r]^{m'(id_G,\gamma)}& \quad m' \circ (id_G \times f \circ m) \ar[d]^{\gamma}\\
f \circ m \circ (+ \times id_P) \ar[rr]_{f(\mu)}& &f \circ m \circ (id_G \times m).
}
\]
\end{defn}

Let $(f,\gamma),(\overline{f},\overline{\gamma}):(P, m, \mu) \rightarrow (P',m',\mu')$ be two morphisms of left $\gic$-torsors.

\begin{defn}
A \textbf{morphism of morphisms of left $G$-torsors} 
$$\varphi: (f,\gamma) \approx (\overline{f},\overline{\gamma})$$
consists of an homotopy $\varphi: f \approx \overline{f}$ which is compatible with the homotopies $\gamma$ and $\overline{\gamma}$, i.e.
the following diagram commute 
 \[\xymatrix{
 m'\circ (id_G \times f) \ar[d]_{m'(id_G,\varphi)}\quad \ar[r]^{\gamma} & f \circ m \ar[d]^{\varphi}\\
m' \circ (id_G \times \overline{f}) \ar[r]_{\overline{\gamma}}& \overline{f} \circ m .
}
\]
\end{defn}

If the complex $G$ acts of the right side instead of the left side, we get the definitions of right $G$-torsor, morphism of right $G$-torsors and morphism of morphisms of right $G$-torsors. 

\begin{defn}\label{def:torcomplexes}
A \textbf{G-torsor} consists of a complex $P=[d^P: P^{-1} \rightarrow P^0]$ of sheaves of sets on $\bS$ (concentrated in degrees -1 and 0) endowed with a structure of left $G$-torsor $(P,m_l,\mu_l)$ and a structure of right $G$-torsor $(P,m_r,\mu_r)$ which are compatible one with another. This compatibility is given by the existence of an homotopy $\kappa : m_l \circ (id_G \times m_r) \approx m_r \circ (m_l \times id_G )$ such that the following diagrams commute 
\[\xymatrix{
  m_l \circ (+ \times m_r) \ar[rr]^{\kappa} \ar[d]_{\mu_l}&  & m_r \circ (m_l \circ(+ \times id_P) \times id_G \ar[d]^{\mu_l})\\
m_l  \circ ( id_G \times m_l \circ (id_G \times m_r)) \ar[dr]_\kappa &  &m_r \circ ( m_l \circ (id_G \times m_l) \times id_G)\\
   & m_l  \circ ( id_G \times m_r \circ (m_l \times id_G )) \ar[ur]_\kappa& 
}
\]
\[\xymatrix{
  m_l \circ (id_G \times m_r \circ (id_P \times +)) \ar[rr]^{\kappa} \ar[d]_{\mu_r}&  & m_r \circ (m_l \times +)   \ar[d]^{\mu_r}\\
m_l \circ (id_G \times m_r \circ (m_r \times id_G))\ar[dr]_\kappa &  &m_r \circ (m_r \circ (m_l \times id_G)\times id_G)\\
   &m_r \circ (m_l \circ (id_G \times m_r) \times id_G) \ar[ur]_\kappa& 
}
\]
\end{defn}

\begin{rem}\label{rem:torcomplexes} If $G=[G^{-1} \stackrel{0}{\rightarrow} G^0]$, then a $G$-torsor consists of a $G^0$-torsor and a $G^{-1}$-torsor.
\end{rem}

\begin{ex} 
The complex $G \in \cK^{[-1,0]}(\bS)$ endowed the morphism of complexes $ +: G \times G \rightarrow G$ is a $G$-torsor. We will call $G$ the trivial $G$-torsor.
\end{ex}

\begin{defn}\label{def:mortorcomplexes}
A \textbf{morphism of $G$-torsors} 
$$(f,\gamma_r,\gamma_l):(P, m_r,m_l,\mu_r,\mu_l,\kappa) \rightarrow (P', m_r',m_l',\mu'_r,\mu'_l,\kappa')$$
 consists of 
\begin{itemize}
	\item a morphism of complexes $f: P \rightarrow P'$, and 
	\item two homotopies
$(\gamma_l): m'_l \circ (id_G \times f) \approx f \circ m_l$ and 
$(\gamma_r): m'_r \circ (f \times id_G ) \approx f \circ m_r$,
\end{itemize}
such that $(f,\gamma_r):(P, m_r,\mu_r) \rightarrow (P', m_r',\mu'_r)$ and $(f,\gamma_l):(P, m_l,\mu_l) \rightarrow (P', m_l',\mu'_l)$ are morphisms of right respectively left $\gic$-torsors, and such that $\gamma_r$ and $\gamma_l$ are compatible with $\kappa$ and $\kappa'$, i.e. the following diagram commutates
\[\xymatrix{
  m_l' \circ ( id_G \times f) \circ( id_G \times m_r) \ar[r]^{\gamma_l}  & f \circ m_l \circ( id_G \times m_r) \ar[r]^{\kappa} & f \circ m_r \circ (m_l \times id_G)   \\
m_l' \circ (id_G \times m'_r \circ (f \times id_G)) \ar[u]^{\gamma_r} \ar[r]^{\kappa'} & m_r' \circ (m'_l \circ (id_G \times f) \times id_G ) \quad \ar[r]^{\gamma_l} &\quad m_r'  \circ (f \circ m_l \times id_G)  \ar[u]^{\gamma_r}.
}
\]
\end {defn}

Let $(f,\gamma_r,\gamma_l),(\overline{f},\overline{\gamma}_r, \overline{\gamma}_l):(P, m_r,m_l,\mu_r,\mu_l,\kappa) \rightarrow (P', m_r',m_l',\mu'_r,\mu'_l,\kappa')$ be two morphisms of $G$-torsors.

\begin{defn}\label{def:mormortorcomplexes}
A \textbf{morphism of morphisms of $G$-torsors} 
$$\varphi: (f,\gamma_r,\gamma_l) \approx (\overline{f},\overline{\gamma}_r, \overline{\gamma}_l)$$
consists of an homotopy $\varphi: f \approx \overline{f}$ such that $\varphi: (f,\gamma_l) \approx (\overline{f}, \overline{\gamma}_l)$ and $\varphi: (f,\gamma_r) \approx (\overline{f},\overline{\gamma}_r)$ are morphisms of morphisms of left respectively right $G$-torsors, i.e. such that $\varphi: f \approx \overline{f}$ is compatible with the homotopies $\gamma_r, \overline{\gamma}_r$ and with the homotopies $\gamma_l, \overline{\gamma}_l.$
\end{defn}

\section{The 2-category of extensions of Picard stacks}

Let $F:\pic \rightarrow \qic$ be an additive functor between strictly commutative Picard $\bS$-stacks. Denote by $\mathbf{1}$
the strictly commutative Picard $\bS$-stack such that for any object $U$ of $\bS$,
$\mathbf{1}(U)$ is the category with one object and one arrow. By \cite{Be10} \S 3
the \textbf{kernel} of $F$, $\ker(F),$ is the fibered product $ \pic \times_\qic \mathbf{1}$ of $\pic$ and $ \mathbf{1}$ over $\qic$ via $F:\pic \rightarrow \qic$ and $ \mathbf{1}: \mathbf{1} \rightarrow \qic$, and the \textbf{cokernel} of $F$, $\coker(F),$ is the fibered sum $\mathbf{1} +^\pic \qic$ of $ \mathbf{1}$ and $\qic$ under $\pic$ via $F: \pic \rightarrow \qic$ and  $\mathbf{1}: \pic \rightarrow \mathbf{1} $.

Let $\pic$ and $\gic$ be two strictly commutative Picard $\bS$-stacks.

\begin{defn}\label{def:ext}
An \textbf{extension} $\eic=(\eic,I,J) $ of $\pic$ by $\gic$
\begin{equation}
\gic \stackrel{I}{\longrightarrow} \eic \stackrel{J}{\longrightarrow} \pic  
\end{equation}
consists of 
\begin{itemize}
	\item a strictly commutative Picard $\bS$-stack $\eic$,
	\item two additive functors $I:\gic \rightarrow \eic$ and $ J:\eic \rightarrow \pic$, and 
	\item an isomorphism of additive functors between the composite $J \circ I$ and the trivial additive functor: $J \circ I \cong 0,$
\end{itemize}
such that the following equivalent conditions are satisfied:
   
   	(a) $\pi_0(J): \pi_0(\eic) \rightarrow \pi_0(\pic)$ is surjective and $I$ induces an equivalence of strictly commutative Picard $\bS$-stacks between $\gic$ and $\ker(J);$
     
      (b) $\pi_1(I): \pi_1(\gic) \rightarrow \pi_1(\eic)$ is injective and $J$ induces an equivalence of strictly commutative Picard $\bS$-stacks between $\coker(I)$ and $\pic$. 
\end{defn}

Let $\pic, \gic, \pic'$ and $\gic'$ be strictly commutative Picard $\bS$-stacks. Let $\eic=(\eic,I,J)$ be an extension of $\pic$ by $\gic$ and let $\eic'=(\eic',I',J')$ be an extension of $\pic'$ by $\gic'$. 

\begin{defn}\label{def:morext}
A \textbf{morphism of extensions}
\[(F,G,H): \eic \longrightarrow \eic' \]
consists of 
\begin{itemize}
	\item three additive functors $F:\eic \rightarrow \eic', G: \pic \rightarrow \pic', H:\gic \rightarrow \gic'$, and 
	\item two isomorphisms of additive functors $J' \circ F \cong G \circ J$ and $ F \circ I \cong I'\circ H$,
\end{itemize}
which are compatible with the isomorphisms of additive functors  $J \circ I \cong 0$ and $J' \circ I' \cong 0$ underlying the extensions $\eic$ and $\eic'$, i.e. the composite 
\[ 0 \stackrel{\cong}{\longleftrightarrow} G \circ 0 \stackrel{\cong}{\longleftrightarrow} G \circ J \circ I  \stackrel{\cong}{\longleftrightarrow}  J' \circ F \circ I \stackrel{\cong}{\longleftrightarrow} J' \circ I'\circ H  \stackrel{\cong}{\longleftrightarrow}  0 \circ H \stackrel{\cong}{\longleftrightarrow} 0\]
should be the identity. 
\end{defn}

Let $(F,G,H),(\overline{F},\overline{G},\overline{H}): \eic \longrightarrow \eic'$ be two morphisms of extensions $\eic=(\eic,I,J)$ of $\pic$ by $\gic$ and $\eic'=(\eic',I',J')$  of $\pic'$ by $\gic'$. 

\begin{defn}\label{def:mormorext}
A \textbf{morphism of morphisms of extensions}
\[(F,G,H): (F,G,H) \Rightarrow (\overline{F},\overline{G},\overline{H}) \]
consists of three morphisms of additive functors $\alpha: F \Rightarrow \overline{F},
\beta : G \Rightarrow \overline{G}$ and $\gamma: H \Rightarrow \overline{H}$ which are compatible with the four isomorphisms of additive functors $J' \circ F \cong G \circ J, F \circ I \cong I'\circ H, J' \circ \overline{F} \cong \overline{G} \circ J$ and $\overline{F} \circ I \cong I'\circ \overline{H}$, i.e. the following diagrams commute for any $g \in \gic$ and $a \in \eic$
 \[\xymatrix{
 FI(g)\ar[d]_{\alpha(I(g))}\quad \ar[r]^{\cong} & I'H(g) \ar[d]^{I'(\gamma(g))}\\
\overline{F}I(g)  \ar[r]_{\cong}& I'\overline{H}(g).
}
\qquad  \qquad
\xymatrix{
 J'F(a)\ar[d]_{J'(\alpha(a))}\quad \ar[r]^{\cong} & GJ(a) \ar[d]^{\beta(J(a))}\\
J'\overline{F}(a)  \ar[r]_{\cong} & \overline{G}J(a).
}
\]
\end{defn}

Extensions of $\pic$ by $\gic$ form a 2-category ${\cExt}(\pic,\gic)$
 where
\begin{enumerate}
	\item the objects are extensions of $\pic$ by $\gic$,
	\item the 1-arrows are morphisms of extensions, 
	\item the 2-arrows are morphisms of morphisms of extensions.
\end{enumerate}

Let $P=[P^{-1} \stackrel{d^P}{\rightarrow}P^0]$ and $G=[G^{-1} \stackrel{d^G}{\rightarrow}G^0]$ be complexes of $\cK^{[-1,0]}(\bS)$ and let
 $F: st(G)  \rightarrow st(P)$ be an additive functor induced by a morphism of complexes $f=(f^{-1},f^0): G \rightarrow P$.
By \cite{Be10} Lemma 3.4,
the strictly commutative Picard $\bS$-stacks $\ker(F)$ and $\coker(F)$
correspond via the equivalence of categories~(\ref{st}) to the following complexes of $\cK^{[-1,0]}(\bS):$
\begin{eqnarray}
\nonumber [\ker(F)] & = & \tau_{ \leq 0}\big( MC(f)[-1]\big) = \big[G^{-1} ~ \stackrel{(f^{-1},-d^G)}{\longrightarrow} ~  \ker(d^P,f^0)\big]\\
\nonumber [\coker(F)] & = & \tau_{\geq -1}MC(f)  = \big[\coker(f^{-1},-d^G) ~ \stackrel{(d^P,f^0)}{\longrightarrow} ~ P^0 \big]
\end{eqnarray}
where $\tau$ denotes the good truncation and $MC(f)$ is the mapping cone of the morphism $f$.
 Therefore we get the following notion of extension for complexes in  $\cK^{[-1,0]}(\bS)$: let $P=[P^{-1} \stackrel{d^P}{\rightarrow}P^0]$ and $G=[G^{-1} \stackrel{d^G}{\rightarrow}G^0]$ be complexes of $\cK^{[-1,0]}(\bS)$.

\begin{defn}\label{def:extcomplexes}
An \textbf{extension} $E=(E,i,j)$ of $P$ by $G$
\[ G \stackrel{i}{\longrightarrow} E \stackrel{j}{\longrightarrow} P \]
consists of 
\begin{itemize}
	\item a complex $E$ of $\cK^{[-1,0]}(\bS)$,
	\item two morphisms of complexes  $i:G \rightarrow E$ and $ j:E \rightarrow P$ of $\cK^{[-1,0]}(\bS)$,
	\item an homotopy between $j \circ i$ and $0$,
\end{itemize}
such that the following equivalent conditions are satisfied:

	(a) ${\h}^{0}(j): {\h}^{0}(E) \rightarrow {\h}^{0}(P)$ is surjective and $i$ induces a quasi-iso\-mor\-phism between $G$ and $ \tau_{\leq 0} (MC(j)[-1])$;
 
      (b) ${\h}^{-1}(i): {\h}^{-1}(G) \rightarrow {\h}^{-1}(E)$ is injective and $j$ induces a quasi-iso\-mor\-phism between $ \tau_{\geq -1} MC(i)$ and $P$.
\end{defn}

\begin{rem}\label{rem:extcomplexes} If $G=[G^{-1} \stackrel{0}{\rightarrow} G^0]$ and $P=[P^{-1} \stackrel{0}{\rightarrow} P^0]$, then an extension of $P$ by $G$ consists of an extension of $P^0$ by $G^0$ and an extension of $P^{-1}$ by $G^{-1}.$
\end{rem}

\begin{rem}\label{rem:exact-ext} Consider a short exact sequence of complexes in $\cK^{[-1,0]}(\bS)$
\[0 \longrightarrow  K \stackrel{i}{\longrightarrow} L \stackrel{j}{\longrightarrow} M \longrightarrow 0.\]
It exists a distinguished triangle $K  \stackrel{i}{\rightarrow}L \stackrel{j}{\rightarrow} M \rightarrow +$ in $\cD(\bS)$, and $M$ is isomorphic to $MC(i)$ in $\cD(\bS)$. Therefore a short exact sequence of complexes 
in $\cK^{[-1,0]}(\bS)$ is an extension of complexes of $\cK^{[-1,0]}(\bS)$ according to the above definition.  
\end{rem}

\begin{rem}\label{rem:Z[I]exttor}
Let $G$ be a complex of $\cK^{[-1,0]}(\bS)$. If $I=[d^I: I^{-1} \rightarrow I^0]$ is a complex of sheaves of 
sets on $\bS$ concentrated in degrees -1 and 0, we denote by  
 ${\ZZ}[I]=[{\ZZ}[d^I]: {\ZZ}[I^{-1}] \rightarrow {\ZZ}[I^0]]$ the complex of abelian sheaves generated by $I$, 
where ${\ZZ}[I^i]$ is the abelian sheaf generated by $I^i$ for $i=-1,0$ (see \cite{SGA4} Expos\'e IV 11). By definition of ${\ZZ}[I]$, the functor
 \[ G \longrightarrow {\Hom}_{\ZZ}({\ZZ}[I],G)\]
 is isomorphic to the functor 
  \[ G \longrightarrow G(I)={\h}^0(I,G_I),\]
where $G_I$ is the fibered product $G \times_{\mathbf{E}} I$, with $\mathbf{E}=[id_\mathbf{e}: \mathbf{e} \rightarrow \mathbf{e}]$ and $\mathbf{e}$ the final object of the category of abelian sheaves on the site $\bS$ (note that $st(E)=\mathbf{1}).$
Taking the respective derived functors, for $i=-1,0,1$ we get the isomorphisms 
\[ {\Ext}^i({\ZZ}[I],G) \cong {\h}^i(I,G_I).\]
Hence by \cite{Be10} Theorem 0.1 and by \cite{B90} Proposition 6.2 we can conclude that the equivalence classes of extensions of 
${\ZZ}[I]$ by $G$ are in bijection with the equivalence classes of $G_I$-torsors over $I$.
\end{rem}

\section{Description of extensions of Picard stacks in terms of torsors}

Let $\pic$ and $\gic$ be two strictly commutative Picard $\bS$-stacks. Denote by $\mathbf{1}$
the strictly commutative Picard $\bS$-stack such that for any object $U$ of $\bS$,
$\mathbf{1}(U)$ is the category with one object and one arrow. Let $\wedge$ be the contracted product of $\gic$-torsors (see 6.7 \cite{B90}). If $K$ is a subset of a finite set $E$, $p_K: \pic^E  \rightarrow \pic^K$ is the projection to the factors belonging to $K$, and $ +_K: \pic^E  \rightarrow \pic^{E-K+1} $ is the group law $+: \pic \times \pic \rightarrow \pic$ on the factors belonging to $K$. If $\iota$ is a permutation of the set $E$, $Perm(\iota): \pic^E  \rightarrow \pic^{\iota(E)}$ is the permutation of the factors  according to $\iota$.
 Moreover let $Sym: \pic \wedge \gic \rightarrow \gic \wedge \pic  $ be the canonical isomorphism that exchange the factors and let $D: \pic  \rightarrow \pic \times  \pic$ be the diagonal morphism.

\begin{thm}\label{thm:ext-tor}
To have an extension $\eic$ of $\pic$ by $\gic$ is equivalent to have 
\begin{enumerate}
	\item a $\gic$-torsor $\eic$ over $\pic$;
	\item a trivialization $I$ of the pull-back $\mathbf{1}^*\eic$ of $\eic$ via the additive functor $\mathbf{1}: \mathbf{1} \rightarrow \pic$, i.e. $I: \gic \rightarrow \mathbf{1}^*\eic$ is an equivalence of $\gic$-torsors between the trivial $\gic$-torsor $\gic$ and the pull-back $\mathbf{1}^*\eic$;
	\item a morphism of $\gic$-torsors over $\pic \times \pic$
	$$M : p_1^* \; \eic \wedge p_2^* \; \eic \longrightarrow +^* \; \eic$$
	 whose restriction over $\mathbf{1} \times \mathbf{1}$ is compatible with the trivialization $I$ (i.e. $M(\mathbf{1}^*\eic,\mathbf{1}^*\eic)=\mathbf{1}^*\eic$);
	\item an isomorphism $\alpha$ of morphisms of $\gic$-torsors over $\pic \times \pic \times \pic$
\begin{equation}\label{eq:alpha}
	\xymatrix{
  p_1^* \; \eic \wedge p_2^* \; \eic \wedge p_3^* \; \eic \ar[r] \ar[d]  &    p_1^* \; \eic \wedge +^*_{23} \; \eic  \ar[d]  \ar@{=>}[dl]^\alpha \\
+^*_{12} \; \eic \wedge  p_3^* \; \eic \ar[r]  & +^*_{123} \; \eic
}  
\end{equation}
	 whose restriction over $\mathbf{1} \times \mathbf{1} \times \mathbf{1} $ is the identity, and whose pull-back over $\pic^4$ via the morphisms cited below satisfies the equality
	 \begin{equation}\label{ext-tor:1}
	 p^*_{123}\; \alpha \circ +^*_{23} \; \alpha \circ p^*_{234} \; \alpha = +^*_{12} \; \alpha \circ +^*_{34}\; \alpha;
	 \end{equation}
	 	\item an isomorphism $\chi: M \Rightarrow Sym \circ M $ of morphisms of $\gic$-torsors over $\pic \times \pic $ 
	\begin{equation}\label{eq:chi}
	\xymatrix{
  p_1^* \; \eic \wedge p_2^* \; \eic  \ar[rr]^{M} \ar[d]_{Sym}  &  &  +^* \; \eic  \\
p_2^* \; \eic \wedge p_1^* \; \eic  \ar[urr]_{M}  \ar@{=>}[ur]^\chi  &  &
}   
\end{equation}
 whose pull-back $D^* \chi$ via the diagonal morphism $D:\pic \rightarrow \pic \times \pic$ is the identity, whose composite with itself $\chi \circ \chi$ is the identity, and whose pull-back over $\pic^3$ via the morphisms quoted below satisfies the equality
	 \begin{equation}\label{ext-tor:2}
	 Perm(132)^*\; \alpha \circ +^*_{23} \; \chi \circ \alpha =  p^*_{13} \; \chi \circ Perm(12)^* \; \alpha \circ p^*_{12} \; \chi. 
	 \end{equation}
\end{enumerate}
\end{thm}

\begin{proof} I) Starting from an extension $\eic=(\eic,I,J)$ of $\pic$ by $\gic$ we will construct the data $\eic,I,M,\alpha,\chi$ given in (1)-(5).
Via the additive functor $I: \gic \rightarrow \eic$, the strictly commutative Picard 
$\bS$-stack $\gic$ acts on the left side and on the right side of $\eic$, furnishing a structure of $\gic$-torsor to $\eic$. Since the additive functor $J: \eic \rightarrow \pic$ induces a surjection $\pi_0(J): \pi_0(\eic) \rightarrow \pi_0(\pic)$ on the $\pi_0$, $\eic$ is in fact a $\gic$-torsor over $\pic$ and so we get (1).
 By definition, $\ker(J)$ is the pull-back $\mathbf{1}^*\eic$ of $\eic$ via $ \mathbf{1}:\mathbf{1} \rightarrow \pic$ and so the condition that $I$ induces an equivalence of strictly commutative Picard $\bS$-stacks between $\gic$ and $\ker(J)$ is equivalent to (2). The existence for any $g \in \gic$ and $a,b \in \eic$ of the associative condition $\sigma: (a+g) +b \cong a+(g +b)$, which  satisfies the pentagonal axiom (\ref{pic-1}), implies that the morphism of $\bS$-stacks
$+: \eic \times \eic \rightarrow \eic$ factorizes through a morphism of $\gic$-torsors over $\pic \times \pic$, $M : p_1^* \; \eic \wedge p_2^* \; \eic \longrightarrow +^* \; \eic$. The neutral object $e$ with its two natural isomorphisms (\ref{eds}) forces the restriction of $M$ over $\mathbf{1} \times \mathbf{1}$ to be compatible with the trivialization $I$. Now the existence for any $a,b, c \in \eic$ of the isomorphism of associativity $\sigma: (a+b) +c \cong a+(b +c)$ implies  the isomorphism $\alpha$ (4). The compatibility of the isomorphism of associativity $\sigma$ with the neutral object (\ref{pic-5}) forces the restriction of $\alpha$ over $\mathbf{1} \times \mathbf{1} \times \mathbf{1} $ to be the identity. Moreover the pentagonal axiom (\ref{pic-1}) satisfied by $\sigma$ is equivalent to the equality (\ref{ext-tor:1}). 
The functorial isomorphism of commutativity $\tau: a+b \cong a+b$ for any $a,b \in \eic$ gives the existence of the isomorphism $\chi$ (5). 
The condition that $\tau_{a,a}$ is the identity for any $ a \in \pic$ (\ref{pic-2}) 
forces the pull-back $D^* \chi$ to be the identity. The coherence condition for $\tau$ (\ref{pic-3}) furnishes that the composite $\chi \circ \chi$ is the identity.
Moreover the hexagonal axiom (\ref{pic-4}) satisfied by $\sigma$ and $\tau$ is equivalent to the equality (\ref{ext-tor:2}). 
\\
II) Now suppose we have the data $\eic,I,M,\alpha,\chi$ given in (1)-(5). We will show that the $\gic$-torsor $\eic$ over $\pic$ is a strictly commutative Picard $\bS$-stack endowed with a structure of extension of $\pic$ by $\gic$. The morphism of $\gic$-torsors over $\pic \times \pic$,
	$M : p_1^* \; \eic \wedge p_2^* \; \eic \longrightarrow +^* \; \eic$
defines a group law $+: \eic \times \eic \rightarrow \eic$ on the $\bS$-stack of groupoids $\eic$. 
The isomorphism $\alpha$ gives the natural isomorphism of associativity $\sigma$ (\ref{ass}). The image of the neutral object of $\gic$ via the trivialization $I: \gic \rightarrow \mathbf{1}^*\eic$ furnishes a neutral object in the pull-back $\mathbf{1}^*\eic$ and so via the projection $\mathbf{1}^*\eic \rightarrow \eic$ we get a neutral object $e$ in $\eic$ (in other words, the neutral object of $\eic$ is the composite $\gic \rightarrow \mathbf{1}^*\eic \rightarrow \eic$). 
The condition $M(\mathbf{1}^*\eic,\mathbf{1}^*\eic)=\mathbf{1}^*\eic$ implies that $e+e \cong e$.
For any $a \in \eic$, the restriction of the morphism of $\gic$-torsors $M$ to $\pic \times \mathbf{1}$ furnishes a $b \in \eic$ and an isomorphism $b+e \cong a$. The restriction of the isomorphism $\alpha$ to $\pic \times \mathbf{1} \times \mathbf{1} $ determines for each $b \in \eic$ an isomorphism of associativity $(b+e)+e \cong b+(e+e)$. Since $e+e \cong e$, for any $a \in \eic$ we get the isomorphism $r_a: a+e \cong a$ (\ref{eds}). In an analogous way we get the natural isomorphism $l_a: e+a \cong a$.
The fact that the restriction of $\alpha$ over $\mathbf{1} \times \mathbf{1} \times \mathbf{1} $ is the identity means that $\sigma$ is compatible with the neutral object $e$ (\ref{pic-5}). Moreover the equality (\ref{ext-tor:1}) satisfied by 
$\alpha$ is equivalent to the pentagonal axiom (\ref{pic-1}) satisfied by $\sigma$.
The isomorphism $\chi$
furnishes the natural isomorphism of commutativity $\tau$ (\ref{com}).
Since the pull-back $D^* \chi$ of $\chi$ via the diagonal morphism $D:\pic \rightarrow \pic \times \pic$ is the identity, $\tau_{a,a}$ is the identity $\forall a \in \pic$ (\ref{pic-2}). The condition  $\chi \circ \chi = id$ implies 
 the coherence condition for $\tau$ (\ref{pic-3}).
 Moreover the equality (\ref{ext-tor:2}) satisfied by 
$\chi$ is equivalent to the hexagonal axiom (\ref{pic-4}) satisfied by $\sigma$ and $\tau$. Now the pull-back $\partial^* M$ of the morphism of $\gic$-torsors $M$ via
 the anti-diagonal morphism $\partial: \pic \rightarrow \pic \times \pic, a \mapsto (-a,a)$ furnishes an isomorphisms of $\gic$-torsors $-^*\eic \wedge \eic \cong \gic$ (here $-:\pic \rightarrow \pic$ is the morphism of $\bS$-stacks underlying $\pic$) and therefore we get a morphism of $\bS$-stacks $-:\eic \rightarrow \eic, a \mapsto -a$ with a natural isomorphism $o_a:a+(-a) \cong e$ (\ref{inv}). The isomorphism $\alpha$ furnishes the second natural isomorphism $c_{ab}:-(a+b) \cong (-a)+(-b)$ of (\ref{inv}).
Until now we have proved that $\eic$ is a strictly commutative Picard $\bS$-stack.\\
If $J: \eic \rightarrow \pic$ denotes the morphism of $\bS$-stacks which furnishes to $\eic$ the structure of torsor over $\pic$, J must be a surjection on the isomorphism classes of objects, i.e. $\pi_0(J): \pi_0(\eic) \rightarrow \pi_0(\pic)$ is surjective. Moreover the compatibility of $J$ with 
the morphism of $\gic$-torsors
	$M : p_1^* \; \eic \wedge p_2^* \; \eic \longrightarrow +^* \; \eic$
implies that $J$ is an additive functor.
As already observed, to have a trivialization $I$ of the pull-back $\mathbf{1}^*\eic$ is equivalent to have an equivalence of strictly commutative Picard $\bS$-stacks between $\gic$ and $\ker(J)$. We still denote $I$ the composite $\gic \cong \mathbf{1}^*\eic \rightarrow \eic$ where the last arrow is the projection $\mathbf{1}^*\eic = \eic \times_\pic \mathbf{1} \rightarrow \eic$. Clearly $I$ is an additive functor. We can conclude that $(\eic, I,J)$ is an extension of $\pic$ by $\gic$.
\end{proof}

As a corollary we get the following statement whose proof is left to the reader

\begin{cor} With the notations of the above Theorem, it exists 
an equivalence of 2-categories between 
the 2-category ${\cExt}(\pic,\gic)$ of extensions of $\pic$ by $\gic$ and 
the 2-category consisting of the data $(\eic,I,M,\alpha,\chi)$.
\end{cor}

Let $G=[d^G: G^{-1} \rightarrow G^0]$ be complexes of $\cK^{[-1,0]}(\bS)$.  If $\mathbf{e}$ denotes the final object of the category of abelian sheaves on the site $\bS$, the complex $\mathbf{E}=[id_\mathbf{e}: \mathbf{e} \rightarrow \mathbf{e}]$ corresponds to the strictly Picard $\bS$-stack $\mathbf{1}$ via the equivalence of category (\ref{st}): $st(\mathbf{E})=\mathbf{1}$.
 Let $P=[d^P:P^{-1} \rightarrow P^0]$ and $ Q=[d^Q: Q^{-1} \rightarrow Q^0]$ are two $G$-torsors the \textbf{contracted product} $P \wedge^G Q$ is the $G$-torsor
 $[d^P \wedge^{d^G} d^Q: P^{-1} \wedge^{G^{-1}} Q^{-1} \rightarrow P^{0} \wedge^{G^0} Q^{0}]$,
where $P^{i} \wedge^{G^i} Q^{i}$ is the contracted product of $P^i$ and $Q^i$ (for $i=-1,0$) and $d^P \wedge^{d^G} d^Q$ is induced by $d^P \times d^Q:P^{-1} \times Q^{-1} \rightarrow P^{0} \times Q^{0}$ (see 1.3 Chapter III \cite{G}).
 If $K$ is a subset of a finite set $F$, $p_K: P^F  \rightarrow P^K$ is the projection to the factors belonging to $K$, and $ +_K: P^F  \rightarrow P^{F-K+1} $ is the group law $+:P \times P \rightarrow P$ on the factors belonging to $K$. If $\iota$ is a permutation of the set $E$, $Perm(\iota): P^F  \rightarrow P^{\iota(F)}$ is the permutation of the factors  according to $\iota$.
 Moreover let $sym: P \wedge P \rightarrow P \wedge P  $ be the canonical isomorphism that exchange the factors and let $d: P \rightarrow P \times  P$ be the diagonal morphism.
As a consequence of \ref{thm:ext-tor} we have the following

\begin{cor}\label{cor:ext-torcomplexes}
To have an extension $E$ of $P$ by $G$ is equivalent to have 
\begin{enumerate}
	\item a $G$-torsor $E$ over $P$;
	\item a trivialization $i$ of the pull-back $\mathbf{1}^*E$ of $E$ via the morphism of complexes $\mathbf{1}: \mathbf{E} \rightarrow P$, i.e. $i: G \rightarrow \mathbf{1}^*E$ is a quasi-isomorphism between the trivial $G$-torsor $G$ and the pull-back $\mathbf{1}^*E$;
	\item a morphism of $G$-torsors over $P \times P$
	$$m : p_1^* \; E \wedge p_2^* \; E \longrightarrow +^* \; E$$
	 whose restriction over $\mathbf{E} \times \mathbf{E}$ is compatible with the trivialization $i$ (i.e. $m(\mathbf{1}^*E,\mathbf{1}^*E)=\mathbf{1}^*E$);
	\item an isomorphism $\alpha$ of morphisms of $G$-torsors over $P \times P \times P$
	\[\xymatrix{
  p_1^* \; E \wedge p_2^* \; E \wedge p_3^* \; E \ar[r] \ar[d]  &    p_1^* \; E \wedge +^*_{23} \; E  \ar[d]  \ar@{~>}[dl]^\alpha \\
+^*_{12} \; E \wedge  p_3^* \; E \ar[r]  & +^*_{123} \; E
}   
\]	
	 whose restriction over $\mathbf{E} \times \mathbf{E} \times \mathbf{E} $ is the identity, and whose pull-back over $P^4$ via the morphisms cited below satisfies the equality
	 \[
	 p^*_{123}\; \alpha \circ +^*_{23} \; \alpha \circ p^*_{234} \; \alpha = +^*_{12} \; \alpha \circ +^*_{34}\; \alpha;
	 \]
	 	\item an isomorphism $\chi: m \approx sym \circ m$ of morphisms of $G$-torsors over $P \times P $ 
	\[\xymatrix{
  p_1^* \; E \wedge p_2^* \; E  \ar[rr]^{m} \ar[d]_{sym}  &  &  +^* \; E  \\
p_2^* \; E \wedge p_1^* \; E  \ar[urr]_{m}  \ar@{~>}[ur]^\chi  &  &
}   
\]
 whose pull-back $d^* \chi$ via the diagonal morphism $d:P \rightarrow P \times P$ is the identity, whose composite with itself $\chi \circ \chi$ is the identity, and whose pull-back over $P^3$ via the morphisms quoted below satisfies the equality
	 \[
	 Perm(132)^*\; \alpha \circ +^*_{23} \; \chi \circ \alpha =  p^*_{13} \; \chi \circ Perm(12)^* \; \alpha \circ p^*_{12} \; \chi. 
	 \]
\end{enumerate}
\end{cor}


\section{The 2-category of biextensions of Picard stacks}

Let $\mathbf{1}$ be the strictly commutative Picard $\bS$-stack 
such that for any object $U$ of $\bS$,
$\mathbf{1}(U)$ is the category with one object and one arrow. 
Let $\gic,\qic$ and $\pic$ be strictly commutative Picard $\bS$-stacks.
We consider on the fibered product $\gic \times_\mathbf{1} \pic$ the structure of 
"strictly commutative Picard $\bS$-stack over $\pic$" of the pull-back $\mathbf{1}^* \gic$ of $\gic$ via the additive functor $\pic \rightarrow \mathbf{1}$
\[    \xymatrix{
 \gic \times_\mathbf{1} \pic  \ar[r]  \ar[d]& \pic \ar[d]^{\mathbf{1}} \\
  \gic \ar[r]_{\mathbf{1}}  & \mathbf{1}.
}\] 
In this case we write $\gic \times_\mathbf{1} \pic =\gic_\pic$. 
On the other hand we can consider on the fibered product $\gic \times_\mathbf{1} \pic$ also the structure of
"strictly commutative Picard $\bS$-stack over $\gic$" of the pull-back $\mathbf{1}^* \pic$ of $\pic$ via the additive functor $\gic \rightarrow \mathbf{1}$.
In this case we write $\gic \times_\mathbf{1} \pic =\pic_\gic$. 
In this section, over $\pic$ we will consider the two strictly strictly commutative Picard $\bS$-stacks $\gic_\pic$ and $\qic_\pic$ and over $\qic$ we will consider the two strictly strictly commutative Picard $\bS$-stacks $\gic_\qic$ and $\pic_\qic$.\\
We identify $\gic_{\pic \times_{\mathbf{1}} \qic }$ as the pull-back
of $\gic_\pic$ via the projection $Pr_1: \pic \times_{\mathbf{1}} \qic \rightarrow \pic$, or as the pull-back
of $\gic_\qic$ via the projection $Pr_2: \pic \times_{\mathbf{1}} \qic \rightarrow \qic$.

Let $\gic,\qic$ and $\pic$ be strictly commutative Picard $\bS$-stacks. 

\begin{defn}\label{def:biext}
 A \textbf{biextension} of $(\pic,\qic)$ by $\gic$ is a $\gic_{\pic \times_{\mathbf{1}} \qic}$-torsor $\bic$ over $\pic \times_{\mathbf{1}} \qic$, endowed with a structure of extension of $\qic_\pic$
by $\gic_\pic$ and a structure of extension of $\pic_\qic$ by $\gic_\qic,$ which are compatible one with another.
\end{defn}

In oder to explain what it means for two extensions to be compatible we used the description of extensions in term of torsors furnished by Theorem \ref{thm:ext-tor}: denote by $(\bic_\pic,I^\qic,M^\qic,\alpha^\qic,\chi^\qic)$ and by
 $(\bic_\qic,I^\pic,M^\pic,\alpha^\pic,\chi^\pic)$ the data corresponding respectively to the extensions $\bic_\pic$ of $\qic_\pic$
by $\gic_\pic$ and $\bic_\qic$ of $\pic_\qic$ by $\gic_\qic$ underlying the biextension $\bic$. In particular, if $p_i^{\qic}: \qic_\pic \times \qic_\pic \rightarrow \qic_\pic $ 
(resp. $p_i^{\pic}: \pic_\qic \times \pic_\qic \rightarrow \pic_\qic $) 
are the projections ($i=1,2$) and $+^{\qic}:\qic_\pic \times \qic_\pic  \rightarrow \qic_\pic $ 
(resp. $+^{\pic}:\pic_\qic \times \pic_\qic  \rightarrow \pic_\qic $)
is the group law of $ \qic_\pic $
(resp. $ \pic_\qic $), 
$$M^\qic: p_1^{\qic \;*} \; \bic_\pic \wedge p_2^{\qic \;*} \; \bic_\pic \longrightarrow +^{\qic \;*} \; \bic_\pic \quad 
(\mathrm{resp.} \; \; M^\pic: p_1^{\pic \;*} \; \bic_\qic \wedge p_2^{\pic \;*} \; \bic_\qic \longrightarrow +^{\pic \;*} \; \bic_\qic)
$$ 
is a morphism of $\gic_\pic$-torsors over $\qic_\pic \times \qic_\pic $ 
(resp. of $\gic_\qic$-torsors over $\pic_\qic \times \pic_\qic$).\\
The two extensions $\bic_\pic$ of $\qic_\pic$ by $\gic_\pic$ 
and $\bic_\qic$ of $\pic_\qic$ by $\gic_\qic$ are \textbf{compatible} if it exists 
an isomorphism $\beta$ of morphisms of $\gic_{\pic \times_{\mathbf{1}} \qic}$-torsors over $(\pic \times_{\mathbf{1}} \qic) \times (\pic \times_{\mathbf{1}} \qic)$
	\begin{equation}
\xymatrix{
	& +^{\pic \;*} p_1^{\qic \;*} \bic \wedge +^{\pic \;*}p_2^{\qic \;*}\bic \ar[dr]^{M^\qic}&   \\
 (p_1^\pic,p_1^\qic)^*\bic \wedge (p_1^\qic,p_2^\pic)^*\bic \wedge
 (p_1^\pic,p_2^\qic)^*\bic \wedge (p_2^\pic,p_2^\qic)^*\bic  \ar[d]_{Sym} \ar[ur]^{M^\pic \wedge M^\pic}  \ar@{=>}[rr]^\beta & & +^{\qic \;*} +^{\pic \;*} \bic\\
(p_1^\pic,p_1^\qic)^*\bic \wedge (p_1^\pic,p_2^\qic)^*\bic\wedge
 (p_1^\qic,p_2^\pic)^*\bic \wedge (p_2^\pic,p_2^\qic)^*\bic \ar[dr]_{M^\qic \wedge M^\qic} & & \\
&+^{\qic \;*} p_1^{\pic \;*}\bic \wedge +^{\qic \;*} p_2^{\pic \;*}\bic \ar[uur]_{M^\pic}& 
}   
\label{eq:compbiextpic}
\end{equation}

 Let $\gic,\qic,\pic,\gic',\qic'$ and $\pic'$ be strictly commutative Picard $\bS$-stacks. Consider a biextension $\bic$ of $(\pic,\qic)$ by $\gic$ and a biextension $\bic'$ of $(\pic',\qic')$ by $\gic'$.

\begin{defn}\label{def:morbiext}
A \textbf{morphism of biextensions}
\[(F,U,V,W):\bic \longrightarrow \bic' \]
\pn consists of
\begin{itemize}
    \item three additive functors $U:\pic \rightarrow \pic',V:\qic \rightarrow \qic',W:\gic \rightarrow \gic'$, and
     \item a morphism of $\bS$-stacks $F:\bic \rightarrow \bic'$,
\end{itemize}
such that $(F,U \times V,U\times W):\bic_\pic \rightarrow \bic'_{\pic'} $ and $(F,U \times V,V\times W):\bic_\qic \rightarrow \bic'_{\qic'} $ are morphisms of extensions. 
\end{defn}
 
 In the above definition we have used the following notation: \\
 $U\times V:\qic_\pic \rightarrow \qic'_{\pic'}, U\times W:\gic_\pic \rightarrow \gic'_{\pic'} , U\times V:\pic_\qic \rightarrow \pic'_{\qic'}  $ and $V\times W: \gic_\qic \rightarrow \gic'_{\qic'}$.
 
Let $(F,U,V,W),(\overline{F},\overline{U},\overline{V},\overline{W}):\bic \longrightarrow \bic'$ be two morphisms of biextensions.

\begin{defn}\label{def:mormorbiext}
A \textbf{morphism of morphisms of biextensions}
\[(\varphi,\alpha,\beta,\gamma):(F,U,V,W) \Rightarrow (\overline{F},\overline{U},\overline{V},\overline{W}) \]
\pn consists of 
\begin{itemize}
    \item three morphisms of additive functors $\alpha: U\times V \Rightarrow \overline{U}\times \overline{V}, \beta: U\times W  \Rightarrow \overline{U}\times \overline{W} $ and $\gamma: V\times W  \Rightarrow \overline{V}\times \overline{W} $,
    \item a morphism of morphisms of $\bS$-stacks $\varphi:F \Rightarrow \overline{F}$,
\end{itemize}
    such that $(\varphi,\alpha,\beta): (F,U \times V,U\times W) \Rightarrow
(\overline{F},\overline{U} \times \overline{V},\overline{U}\times \overline{W})$ and $ (\varphi,\alpha,\gamma):   (F,U \times V,V\times W) \Rightarrow (\overline{F},\overline{U} \times \overline{V},\overline{V}\times \overline{W})$ are morphisms of morphisms of extensions.
\end{defn}

 Biextensions of $(\pic,\qic)$ by $\gic$ form a 2-category ${\cBiext}(\pic,\qic;\gic)$
 where
\begin{enumerate}
	\item the objects are biextensions of $(\pic,\qic)$ by $\gic$,
	\item the 1-arrows are morphisms of biextensions, 
	\item the 2-arrows are morphisms of morphisms of biextensions.
\end{enumerate}

We have the following equivalence of 2-categories
\[ {\cBiext}(\pic,{\mathbf{1}};\gic) \cong {\cBiext}({\mathbf{1}}, \pic ;\gic) \cong 
{\cExt}( \pic ,\gic). \]

Let $P=[d^P: P^{-1} \rightarrow P^0], Q=[d^Q: Q^{-1} \rightarrow Q^0]$ and $G=[d^G: G^{-1} \rightarrow G^0]$ be complexes of $\cK^{[-1,0]}(\bS)$. If $\mathbf{e}$ denotes the final object of the category of abelian sheaves on the site $\bS$, the complex $\mathbf{E}=[id_\mathbf{e}: \mathbf{e} \rightarrow \mathbf{e}]$ corresponds to the strictly Picard $\bS$-stack $\mathbf{1}$ via the equivalence of category (\ref{st}): $st(\mathbf{E})=\mathbf{1}$. We denote by $G_P$ (resp. $P_Q,G_Q,G_{P \times_{\mathbf{E}} Q}$) the fibered product $G \times_{\mathbf{E}} P$ (resp. $P \times_{\mathbf{E}} Q,G \times_{\mathbf{E}} Q,G \times_{\mathbf{E}} P \times_{\mathbf{E}} Q$).

\begin{defn}\label{def:biextcomplexes}
An \textbf{biextension} of $(P,Q)$ by $G$ is a $G_{P \times_{\mathbf{E}} Q}$-torsor $B$ over $P \times_{\mathbf{E}} Q$, endowed with a structure of extension of $Q_P$ by $G_P$ and a structure of extension of $P_Q$ by $G_Q,$ which are compatible one with another.
\end{defn}

\begin{rem}\label{rem:biextcomplexes} Because of Remarks (\ref{rem:torcomplexes}) and (\ref{rem:extcomplexes}), if $G=[G^{-1} \stackrel{0}{\rightarrow} G^0], P=[P^{-1} \stackrel{0}{\rightarrow} P^0]$ and $Q=[Q^{-1} \stackrel{0}{\rightarrow} Q^0]$,
then a biextension of $(P,Q)$ by $G$ consists of a biextension of $(P^0,Q^0)$ by $G^0$ and a biextension of $(P^{-1},Q^{-1})$ by $G^{-1}$.
\end{rem}

 In oder to explain what it means for two extensions to be compatible we used the description of extensions in term of torsors furnished by Corollary \ref{cor:ext-torcomplexes}:
 denote by $(B_P,i^Q,m^Q,\alpha^Q,\chi^Q)$ and by
 $(B_Q,i^P,m^P,\alpha^P,\chi^P)$ the data corresponding respectively to the extensions $B_P$ of $Q_P$
by $G_P$ and $B_Q$ of $P_Q$ by $G_Q$ underlying the biextension $B$. In particular, if $p_i^{Q}: Q_P \times Q_P \rightarrow Q_P $ 
(resp. $p_i^{P}: P_Q \times P_Q \rightarrow P_Q $) 
are the projections ($i=1,2$) and $+^{Q}:Q_P \times Q_P  \rightarrow Q_P $ 
(resp. $+^{P}:P_Q \times P_Q  \rightarrow P_Q $)
is the group law of $ Q_P $
(resp. $ P_Q $), 
$$m^Q: p_1^{Q \;*} \; B_P \wedge p_2^{Q \;*} \; B_P \longrightarrow +^{Q \;*} \; B_P \quad 
(\mathrm{resp.} \; \; m^P: p_1^{P \;*} \; B_Q \wedge p_2^{P \;*} \; B_Q \longrightarrow +^{P \;*} \; B_Q)
$$ 
is a morphism of $G_P$-torsors over $Q_P \times Q_P $ 
(resp. of $G_Q$-torsors over $P_Q \times P_Q$).\\
The two extensions $B_P$ of $Q_P$ by $G_P$ 
and $B_Q$ of $P_Q$ by $G_Q$ are \textbf{compatible} if it exists 
an isomorphism $\beta$ of morphisms of $G_{P \times_{\mathbf{E}} Q}$-torsors over $(P \times_{\mathbf{E}} Q) \times (P \times_{\mathbf{E}} Q)$
\begin{equation}
\xymatrix{
	& +^{P \;*} p_1^{Q \;*} B \wedge +^{P \;*}p_2^{Q \;*}B \ar[dr]^{m^Q}&   \\
 (p_1^P,p_1^Q)^*B \wedge (p_1^Q,p_2^P)^*B \wedge
 (p_1^P,p_2^Q)^*B \wedge (p_2^P,p_2^Q)^*B  \ar[d]_{sym} \ar[ur]^{m^P \wedge m^P}  \ar@{~>}[rr]^\beta & & +^{Q \;*} +^{P \;*} B\\
(p_1^P,p_1^Q)^*B \wedge (p_1^P,p_2^Q)^*B\wedge
 (p_1^Q,p_2^P)^*B \wedge (p_2^P,p_2^Q)^*B \ar[dr]_{m^Q \wedge m^Q} & & \\
&+^{Q \;*} p_1^{P \;*}B \wedge +^{Q \;*} p_2^{P \;*}B \ar[uur]_{m^P}& 
}   
\label{eq:compbiextcomp}
\end{equation}

\section{Operations on biextensions of strictly commutative Picard stacks}

Let $U:\pic' \rightarrow \pic,V:\qic' \rightarrow \qic,W:\gic \rightarrow \gic'$
be three additive functors. Consider a biextension $\bic$ of $(\pic,\qic)$ by $\gic$.

\begin{defn} The \emph{pull-back $(U\times V)^*\eic$ of the biextension $\bic$ via the additive functors $U \times V:\pic' \times_{\mathbf{1}} \qic' \rightarrow \pic \times_{\mathbf{1}} \qic$} is the fibered product $\bic \times_{\pic \times_{\mathbf{1}} \qic} (\pic' \times_{\mathbf{1}} \qic')$
of $\bic$ and $\pic' \times_{\mathbf{1}} \qic'$ over $\pic \times_{\mathbf{1}} \qic $ via $U\times V$.
\end{defn}

By \cite{Be10} Lemma 4.2 the pull-back $(U\times V)^*\bic$ is a biextension 
of $(\pic',\qic')$ by $\gic$.

\begin{defn} The \emph{push-down $W_*\bic$ of the biextension $\bic$ via the additive functor $W:\gic \rightarrow \gic'$} is the fibered sum $\bic +^\gic \gic'$ of $\bic$ and $\gic'$ under $\gic$ via $W$.
\end{defn}

By \cite{Be10} Lemma 4.4 the push-down $W_*\bic$ is a biextension 
of $(\pic,\qic)$ by $\gic'$.

Now let $\bic'$ be another biextension of $(\pic,\qic)$ by $\gic$. According to \cite{Be10} Lemma 4.5, the product $\bic \times \bic'$ is a biextension of $(\pic \times \pic, \qic \times \qic)$ by $\gic \times \gic$.

\begin{defn} \label{def:+} 
 The \emph{sum $\bic + \bic'$ of the biextensions $\bic$ and $\bic'$} is the following biextension of $(\pic,\qic)$ by $\gic$
\begin{equation}
D^* +_* ( \bic \times \bic')
\label{eq:+}
\end{equation}
 where $+: \gic \times \gic \rightarrow \gic$ is the group law of $\gic$ and 
 $D=(D_\pic,D_\qic): \pic \times \qic \rightarrow (\pic \times \pic) \times (\qic \times \qic)$ with $D_\pic$ (resp. $D_\qic$) the diagonal functor of $\pic$ (resp. $\qic$).
\end{defn}

As a consequence of \cite{Be10} Lemma 4.7 we have the following

\begin{lem}\label{lem:sumofbiext}
The above notion of sum of biextensions defines on the set of equivalence classes of biextensions of $(\pic,\qic)$ by $\gic$ an associative, commutative group law with neutral object, that we denote $\gic \times_{\mathbf{1}} \pic \times_{\mathbf{1}} \qic$.
\end{lem}

Remark that the neutral object is the trivial $\gic_{\pic \times_{\mathbf{1}} \qic}$-torsor over $\pic \times_{\mathbf{1}} \qic$.

\section{Proof of theorem 0.1 (\textbf{b}) and (\textbf{c})}

Let $\pic,\qic$ and $\gic$ be three strictly commutative Picard $\bS$-stacks.
According to Lemma \ref{lem:sumofbiext}, the set of equivalence classes of objects of ${\cBiext}(\pic,\qic;\gic)$ is a commutative group with neutral object $\bic_0= \gic \times_{\mathbf{1}} \pic \times_{\mathbf{1}} \qic$. We denote this group by
$${\cBiext}^1(\pic,\qic;\gic).$$

The monoid of isomorphism classes of arrows from an object $\bic$ of ${\cBiext}(\pic,\qic;\gic)$ to itself is canonically isomorphic to the monoid of isomorphism classes of arrows from $\bic_0$ to itself: to an isomorphism class of an arrow $F :\bic_0 \rightarrow \bic_0$ the canonical isomorphism associates the isomorphism class of the arrow $F + Id_{\bic}$ from $\bic_0+ \bic \cong \bic$ to itself. 
The monoid of isomorphism classes of arrows from $\bic_0$ to itself is a commutative group via the composition law $(\overline{F},\overline{G}) \mapsto \overline{F+G}$
(here $\overline{F+G}$ is the isomorphism class of the arrow $F+G$ from $\bic_0 +\bic_0 \cong \bic_0$ to itself). Hence we can conclude that the set of isomorphism classes of arrows from an object of ${\cBiext}(\pic,\qic;\gic)$ to itself is a commutative group that we denote by 
$${\cBiext}^0(\pic,\qic;\gic).$$

The monoid of automorphisms of arrows from an object $\bic$ of ${\cBiext}(\pic,\qic;\gic)$ to itself is canonically isomorphic to the monoid of automorphisms of arrows from $\bic_0$ to itself: to an automorphism $\alpha:F \Rightarrow F$ of an arrow $F :\bic_0 \rightarrow \bic_0$ the canonical isomorphism associates the automorphism $\alpha + id_{Id_\bic}:F + Id_\bic \Rightarrow F + Id_\bic $ of the arrow  $F + Id_{\bic}$ from $\bic_0+ \bic \cong \bic$ to itself. 
The monoid of automorphisms of arrows from $\bic_0$ to itself is a commutative group via the following composition law: if $\alpha:F \Rightarrow F$ and $\beta:G \Rightarrow G$, then $\alpha +\beta:F+G \Rightarrow F+G$, with $F+G$ an arrow from $\bic_0 +\bic_0 \cong \bic_0$ to itself. Hence we can conclude that the set of automorphisms of an arrow from an object of ${\cBiext}(\pic,\qic;\gic)$ to itself is a commutative group that we denote by
$${\cBiext}^{-1}(\pic,\qic;\gic).$$

\emph{Proof of Theorem \ref{intro:mainthm}} (\textbf{b}) and (\textbf{c}). As we have observed at the beginning of this section, in order to prove (\textbf{b}) and (\textbf{c}) we can work with the biextension $\bic_0= \gic \times_{\mathbf{1}} \pic \times_{\mathbf{1}} \qic$ of $(\pic,\qic)$ by $\gic$. In particular $\bic_0$ is a strictly commutative Picard $\bS$-stack and so the group of isomorphism classes of arrows from $\bic_0$ to itself is the cohomology group ${\h}^0([{\HOM}(\bic_0,\bic_0)])$ and the group 
of automorphisms of arrows from $\bic_0$ to itself is the cohomology group ${\h}^{-1}([{\HOM}(\bic_0,\bic_0)])$. Therefore, in order to conclude it is enough to compute the complex $[{\HOM}(\bic_0,\bic_0)]$.\\
Let $F: \bic_0\rightarrow \bic_0$ be an additive functor.
Since $F$ is first of all an arrow from the $\gic_{\pic \times_{\mathbf{1}} \qic}$-torsor over $\pic \times_{\mathbf{1}} \qic$ underlying $\bic_0$ to itself,  $F$ is given by the formula
$$F(b)= b + IF'J (b) \qquad \forall \; b \in \bic_0$$
where $F': \pic \times \qic \rightarrow \gic$ is an additive functor and $J: \bic_0 \rightarrow \pic \times \qic$ and $I: \gic \rightarrow \bic_0$ the additive functors underlying the structure of $\gic_{\pic \times_{\mathbf{1}} \qic}$-torsor over $\pic \times_{\mathbf{1}} \qic$ of $\bic_0$.
Now $F: \bic_0 \rightarrow \bic_0$ must be compatible with the structures of extension of $\qic_\pic$
by $\gic_\pic$ and of extension of $\pic_\qic$ by $\gic_\qic$ underlying $\bic_0$,
and so $F': \pic \times \qic \rightarrow \gic$ must be a biadditive functor. 
Hence we get that ${\HOM}(\bic_0,\bic_0)$ is equivalent as strictly commutative Picard $\bS$-stack to ${\HOM}(\pic, \qic;\gic)$ via the following additive functor 
\begin{eqnarray}
\nonumber {\HOM}(\pic, \qic;\gic) &\longrightarrow & {\HOM}(\bic_0,\bic_0) \\
 \nonumber  F' & \mapsto & \big(b \mapsto b + IF'J (b) \big).
\end{eqnarray}
In the example \ref{ex:pic} we have observed that the strictly commutative Picard $\bS$-stacks ${\HOM}(\pic, \qic;\gic)$ and ${\HOM}(\pic \otimes \qic,\gic)$
are equivalent as strictly commutative Picard $\bS$-stacks and so 
 $$[{\HOM}(\bic_0,\bic_0)]= \tau_{\leq 0}{\R}{\Hom}\Big(\tau_{\geq -1}([\pic]\otimes^{\LL}[\qic]),[\gic] \Big),$$ 
 i.e. the group of isomorphism classes of additive functors from $\bic_0$ to itself is isomorphic to the group ${\Hom}_{\cD(\bS)}([\pic]\otimes^{\LL}[\qic],[\gic])$, and the 
 group of automorphisms of an additive functor from $\bic_0$ to itself is isomorphic to the group ${\Hom}_{\cD(\bS)}([\pic]\otimes^{\LL}[\qic],[\gic][-1])$. \\

In Section 10 we gives another proof of Theorem \ref{intro:mainthm} \textbf{b} and \textbf{c}.

\section{The 2-category $\Psi_{\lic^.}(\gic)$ and its homological interpretation}

A cochain complex of strictly commutative Picard $\bS$-stacks
\[\stackrel{}{\longrightarrow}  \lic^{-1} \stackrel{D^{-1}}{\longrightarrow} \lic^0 \stackrel{D^0}{\longrightarrow} \lic^1 \stackrel{D^1}{\longrightarrow}\]
consists of 
\begin{itemize}
	\item strictly commutative Picard $\bS$-stacks $\lic^i$ for $i \in \ZZ$,
	\item additive functors $D^i$ for $i \in \ZZ$,
	\item isomorphisms of additive functors between the composite $D^{i+1} \circ D^i$ and the trivial additive functor: $D^{i+1} \circ D^i \cong 0$ for $i \in \ZZ$.
\end{itemize}

Let $\gic$ be a strictly commutative Picard $\bS$-stack and let
\[\lic^.: \qquad \ric \stackrel{D^\ric}{\longrightarrow} \qic \stackrel{D^\qic}{\longrightarrow} \pic \stackrel{D^\pic}{\longrightarrow} 0\] 
be a complex of strictly commutative Picard $\bS$-stacks with $\pic, \qic$ and $\ric$ in degrees 0,-1 and -2 respectivelly.

\begin{defn}\label{psi}
Denote by $\Psi_{\lic^.}(\gic)$ the 2-category 
\begin{enumerate}
       \item whose objects are pairs $(\eic,I)$ with $\eic$ an extension of $\pic$ by $\gic$ and $I$ a trivialization of the extension $(D^\qic)^* \eic$ of $\qic$ by $\gic$ obtained as pull-back of $\eic$ by $D^\qic$.
        Moreover we require that the corresponding trivialization
           $(D^\ric)^* I$ of $(D^\ric)^* (D^\qic)^* \eic$ is the trivialization arising from the isomorphism of transitivity $ (D^\ric)^* (D^\qic)^* \eic \cong  (D^\qic \circ D^\ric)^*\eic$ and the relation $ D^\qic \circ D^\ric \cong 0$. Note that to have such a trivialization $I$ is the same thing as to have a lifting $I: \qic \rightarrow \eic$ of $D^\qic:\qic  \rightarrow \pic$ such that $ I \circ D^\ric \cong 0;$
      \item whose 1-arrows $F: (\eic,I) \rightarrow (\eic',I')$ are morphisms of extensions $F: \eic \rightarrow \eic'$ compatible with the trivializations $I,I'$, i.e. we have an isomorphism of additive functors $F \circ I \cong I'$,
     \item whose 2-arrows $\alpha: F \Rightarrow \overline{F}$ are morphisms of morphisms of extensions which are compatible with the isomorphisms of additive functors $F \circ I \cong I'$ and $\overline{F} \circ I \cong I'$, i.e. the following diagram commutes for any $q \in \qic$
 \[\xymatrix{
 FI(q)\ar[d]_{\alpha(I(q))}\quad \ar[r]^{\cong} & I'(q) \\
\overline{F}I(q).\ar[ur]_{\cong} & 
}
\]    
\end{enumerate}
 \end{defn}

We can summarize the data $(\eic,I)$ with the following diagram:
\[\begin{array}{ccccccccc}
  \gic & =& \gic& =& \gic & & \\
 \downarrow && \downarrow && \downarrow &&  \\
(D^\ric)^*(D^\qic)^*\eic  & \rightarrow &(D^\qic)^*\eic & \rightarrow&\eic&& \\
 \downarrow && I\uparrow\downarrow && \downarrow &&  \\
  \ric& \stackrel{D^\ric}{\rightarrow}&\qic&
\stackrel{D^\qic}{\rightarrow}&\pic& \rightarrow &0 \\
\end{array}\]

The sum of extensions of strictly commutative Picard $\bS$-stacks defined in \cite{Be10} 4.6 furnishes a group law on the set of equivalence classes of objects of $\Psi_{\lic^.}(\gic)$. We denote this group by $\Psi_{\lic^.}^1(\gic).$
The neutral object of $\Psi_{\lic^.}^1(\gic)$ is the object $(\eic_0,I_0)$ where $\eic_0$ is the extension $\gic \times_{\mathbf{1}} \pic$ of $\pic$ by $\gic$ and $I_0$ is the trivialization $(Id_{\qic},0)$
of the extension $(D^\qic)^*\eic_0= \gic \times_{\mathbf{1}} \qic$ of $\qic$ by $\gic$. We can consider $I_0$ as the lifting $(D^\qic,0)$ of $D^\qic: \qic \rightarrow \pic$.

The monoid of isomorphism classes of arrows from an object $(\eic, I)$ of $\Psi_{\lic^.}(\gic)$ to itself is canonically isomorphic to the monoid of isomorphism classes of arrows from $(\eic_0,I_0)$ to itself: to an isomorphism class of an arrow $F :(\eic_0,I_0) \rightarrow (\eic_0,I_0)$ the canonical isomorphism associates the isomorphism class of the arrow $F + Id_{(\bic,I)}$ from $(\eic_0,I_0)+ (\eic,I) \cong (\eic,I)$ to itself. 
The monoid of isomorphism classes of arrows from $(\eic_0,I_0)$ to itself is a commutative group via the composition law $(\overline{F},\overline{G}) \mapsto \overline{F+G}$
(here $\overline{F+G}$ is the isomorphism class of the arrow $F+G$ from $(\eic_0,I_0) +(\eic_0,I_0) \cong (\eic_0,I_0)$ to itself). Hence we can conclude that the set of isomorphism classes of arrows from an object of $\Psi_{\lic^.}(\gic)$ to itself is a commutative group that we denote by 
$\Psi_{\lic^.}^0(\gic).$

The monoid of automorphisms of arrows from an object $(\eic, I)$ of $\Psi_{\lic^.}(\gic)$ to itself is canonically isomorphic to the monoid of automorphisms of arrows from $(\eic_0,I_0)$ to itself: to an automorphism $\alpha:F \Rightarrow F$ of an arrow $F :(\eic_0,I_0) \rightarrow (\eic_0,I_0)$ the canonical isomorphism associates the automorphism $\alpha + id_{Id_{(\eic,I)}}:F + Id_{(\eic,I)} \Rightarrow F + Id_{(\eic,I)} $ of the arrow  $F + Id_{(\eic,I)}$ from $(\eic_0,I_0) + (\eic,I) \cong (\eic,I)$ to itself.
The monoid of automorphisms of arrows from $(\eic_0,I_0)$ to itself is a commutative group via the following composition law: if $\alpha:F \Rightarrow F$ and $\beta:G \Rightarrow G$, then $\alpha +\beta:F+G \Rightarrow F+G$, with $F+G$ an arrow from $(\eic_0,I_0) +(\eic_0,I_0) \cong (\eic_0,I_0)$ to itself. Hence we can conclude that the set of automorphisms of an arrow from an object of $\Psi_{\lic^.}(\gic)$ to itself is a commutative group that we denote by
$\Psi_{\lic^.}^{-1}(\gic).$

If $[\ric]=[d^R:R^{-1} \rightarrow R^0], [\pic]=[d^P:P^{-1} \rightarrow P^0]$ and $[\qic]=[d^Q:Q^{-1} \rightarrow Q^0]$, the complex $\lic^.$ of strictly commutative Picard $\bS$-stacks furnishes, modulo quasi-isomorphisms, a diagram in the category $\cK(\bS)$ of complexes of abelian sheaves
 \[
[\lic^.]: \qquad \qquad R  \stackrel{D^R}{\longrightarrow}  Q \stackrel{D^Q}{\longrightarrow} P \longrightarrow 0
\] 
where $D^R=(d^{R,-1},d^{R,0}), D^Q=(d^{Q,-1},d^{Q,0})$ and $ D^R \circ D^Q$ is homotopic to zero. We can consider $[\lic^.]$ as a bicomplex of abelian sheaves,   
\[\xymatrix{
 R^{-1} \ar[d]_{d^R} \ar[r]^{d^{R, -1}} & Q^{-1} \ar[d]^{d^Q} \ar[r]^{d^{Q, -1}} & P^{-1} \ar[d]^{d^P} \ar[r] &0\\
R^0 \ar[r]^{d^{R, 0}} &  Q^0  \ar[r]^{d^{Q, 0}} & P^0  \ar[r] &0
}
\]  
where $P^0,P^{-1},Q^0,Q^{-1},R^0,R^{-1}$ are respectively in degrees $(0,0),(0,-1),(-1,0),\\(-1,-1),(-2,0),(-2,-1)$. Denote by ${\Tot}([\lic^.])$ the total complex of this bicomplex. 
We have the following homological interpretation of the groups $\Psi_{\lic^.}^i(\gic)$.

\begin{thm}\label{thm:psi-ext}
$$\Psi_{\lic^.}^i(\gic) \cong {\Ext}^i\big({\Tot}([\lic^.]),[\gic]\big)= {\Hom}_{\cD(\bS)}\big({\Tot}([\lic^.]),[\gic][i]\big) \qquad \qquad i=-1,0,1.$$
\end{thm}

\emph{Proof of the cases i=-1 and 0.} As observed above, $\Psi_{\lic^.}^0(\gic)$ is canonically isomorphic to the group of isomorphism classes of arrows from $(\eic_0,I_0)$ to itself, and $\Psi_{\lic^.}^{-1}(\gic)$ is canonically isomorphic to the group of automorphisms of arrows from $(\eic_0,I_0)$ to itself. This implies that in order to prove the cases $i=-1,0$ we can work with the neutral object $(\eic_0,I_0)$. By definition of 1-arrows in the 2-category $\Psi_{\lic^.}(\gic)$, the additive functor $F: \eic_0 \rightarrow \eic_0$ is a 1-arrow from $(\eic_0,I_0)$ to itself if we have an isomorphism of additive functors  $F \circ D^\qic \cong 0$, i.e. if $F$ is an object of the strictly commutative Picard $\bS$-stack
\[\kic=\ker\big({\HOM}(\pic,\gic)  \stackrel{D^\qic}{\rightarrow} {\HOM}(\qic,\gic)\big).\]
Therefore we have the equalities 
\begin{equation}
\Psi_{\lic^.}^i(\gic) = {\h}^i\big([\kic]\big) \qquad \qquad i=-1,0
\label{eq:kic}
\end{equation}
and in order to conclude, it is enough to compute the complex $[\kic]$ of $\cK^{[-1,0]}(\bS).$
By \cite{Be10} Lemma 3.4 we have  
\[ [\kic]= \tau_{\leq 0} \Big( MC\big(\tau_{\leq 0}{\R}{\Hom}([\pic],[\gic]) \stackrel{(d^{R, -1},d^{R, 0})}{\longrightarrow} 
\tau_{\leq 0}{\R}{\Hom}([\qic],[\gic]) \big)[-1] \Big).\]
Explicitly, if $[\gic]=[d^G:G^{-1} \rightarrow G^0]$ we get 
\begin{equation}
[\kic]= \big[{\Hom}(P^0,G^{-1}) \stackrel{((d^G,d^P),d^{Q, 0})}{\longrightarrow} 
 K_1 +K_2\big]
\label{eq:psi0-1}
\end{equation}
 where 
\[
\begin{aligned}
\nonumber	K_1&= \ker\big({\Hom}(P^0,G^0) + {\Hom}(P^{-1},G^{-1}) \stackrel{(d^{Q, 0},d^{Q, -1})}{\rightarrow} {\Hom}(Q^0,G^0)+{\Hom}(Q^{-1},G^{-1}) \big)\\ 
\nonumber	K_2&=\ker \big({\Hom}(Q^0,G^{-1}) \stackrel{(d^G,d^Q)}{\rightarrow} {\Hom}(Q^0,G^0)+{\Hom}(Q^{-1},G^{-1}) \big).
\end{aligned}
\]
In order to simplify notation let $L^.: L^{-3} \rightarrow L^{-2} \rightarrow L^{-1} \rightarrow L^0$ be the total complex ${\Tot}([\lic^.])$. In particular $L^0=P^0, L^{-1}= P^{-1}+Q^0$ and $L^{-2}=Q^{-1}+R^0.$ 
The stupid filtration of the complexes $L^.$ and $G$ furnishes the spectral sequence
\begin{equation}
{\E}^{pq}_1= \bigoplus_{p_2-p_1=p} {\Ext}^q (L^{p_1},G^{p_2}) \Longrightarrow  {\Ext}^* (L^.,G).
\label{eq:suitesp}
\end{equation} 
This spectral sequence is concentrated in the region of the plane defined by $ -1 \leq p \leq 3$ and $q \geq 0$. We are interested on the total degrees -1 and 0. The rows $q=1$ and $q=0$ are
\small
\[
\begin{aligned}
& {\Ext}^1(L^{0},G^{-1}) \rightarrow  {\Ext}^1(L^{0},G^{0})  \oplus {\Ext}^1(L^{-1},G^{-1}) \rightarrow  {\Ext}^1(L^{-1},G^{0})  \oplus {\Ext}^1(L^{-2},G^{-1}) \rightarrow ...\\
&{\Hom}(L^{0},G^{-1})\stackrel{d_1^{-10}}{\rightarrow} {\Hom}(L^{0},G^{0})  \oplus {\Hom}(L^{-1},G^{-1})  \stackrel{d_1^{00}}{\rightarrow}  {\Hom}(L^{-1},G^{0})  \oplus {\Hom}(L^{-2},G^{-1}) \rightarrow ...
\end{aligned}
\]
\normalsize
Since ${\Ext}^1(L^{0},G^{-1})=0,$ i.e. the only extension of $[G^{-1} \rightarrow 0]$ by $[0 \rightarrow L^0]$ is the trivial one, we obtain  
\begin{eqnarray}
 \nonumber  {\Hom}_{{\cD}({\bS})}(L^.,G[-1]) &= & {\Ext}^{-1}(L^.,G) = {\E}^{-10}_2 = \ker ( d_1^{-10} ),\\
 \label{eq:D(S)=H(S)}  {\Hom}_{{\cD}({\bS})}(L^.,G) &= & {\Ext}^0(L^.,G) = {\E}^{00}_2 = \ker ( d_1^{00} ) / \im(d_1^{-10}) .
 \end{eqnarray}
 Comparing the above equalities with the explicit computation (\ref{eq:psi0-1}) of the complex $[\kic]$, we get 
 \[ {\Ext}^i(L^.,G) = {\h}^i\big([\kic]\big) \qquad \qquad i=-1,0. \]
 These equalities together with equalities (\ref{eq:kic}) give the expected statement.
 
 \begin{rem}
 In the computation (\ref{eq:psi0-1}) the term ${\Hom}(P^{-1},G^0)$ does not appear because we work with the good truncation $\tau_{\leq 0}{\R}{\Hom}([\pic],[\gic])$. In the spectral sequence (\ref{eq:suitesp}) this term appear but we are interested in elements which become zero in ${\Hom}(P^{-1},G^0)$. 
 \end{rem}

\begin{rem}\label{rem:D(S)=H(S)} If ${\cH}(\bS)$ denotes the category of complexes of abelian sheaves on $\bS$ modulo homotopy, by equality (\ref{eq:D(S)=H(S)}) we have $ {\Hom}_{{\cD}({\bS})}(L^.,G)={\Hom}_{{\cH}({\bS})}(L^.,G).$
\end{rem}

\emph{Proof of the case i=1.} First we show how an object $(\eic,I)$ of $\Psi_{\lic^.}(\gic)$ defines a morphism ${\Tot}([\lic^.]) \rightarrow [\gic][1]$ in the derived category $\cD(\bS)$. Recall that $\eic$ is an extension of $\pic$ by $\gic$. Denote $J:\eic \rightarrow \pic$ the additive functor underlying the extension $\eic$. Since the trivialization $I$ can be seen as a lifting $ \qic \rightarrow \eic$ of $D^\qic:\qic  \rightarrow \pic$ such that $ I \circ D^\ric \cong 0$, the diagram of additive functors 
\[
\xymatrix{
  \ric  \ar[d] \ar[r]^{D^\ric} & \qic \ar[d]_{I} \ar[r]^{D^\qic} & \pic  \ar[d]^{Id_\pic} \ar[r]  & 0 \\
0 \ar[r] & \eic  \ar[r]^{J} & \pic  \ar[r] & 0
}
\]
commutes. It furnishes, modulo quasi-isomorphisms, a diagram in the category $\cK(\bS)$ of complexes of abelian sheaves on $\bS$ 
 \begin{equation}\label{eq:su}
\xymatrix{
[\lic^.]: & R  \ar[d] \ar[r]^{D^R} & Q \ar[d]_{i} \ar[r]^{D^Q} & P \ar[d]^{id_P} \ar[r]  & 0 \\
MC(j): &0 \ar[r] & E  \ar[r]^{j} & P  \ar[r] & 0
}
\end{equation}
where $E=[\eic] \in \cK^{[-1,0]}(\bS)$, $D^R=(d^{R,-1},d^{R,0}), D^Q=(d^{Q,-1},d^{Q,0})$, $ i \circ D^R$ is homotopic to zero and $ j \circ i$ is homotopic to $ id_P \circ D^Q$. Putting the complex $P$ in degree 0, 
the above diagram gives an arrow  
\[ c(\eic,I): {\Tot}([\lic^.]) \longrightarrow MC(j)\]
in the derived category $\cD(\bS)$. Since $\gic$ is equivalent as strictly commutative Picard $\bS$-stack to $\ker(J)$, i.e. $[\gic]$ is quasi-isomorphic to $\tau_{\leq 0}(MC(j)[-1])$, we have constructed a canonical arrow
  \begin{eqnarray}\label{c}
 c: \Psi_{\lic^.}^1(\gic) & \longrightarrow & {\Hom}_{\cD(\bS)}\big({\Tot}([\lic^.]),[\gic][1]\big)  \\
 \nonumber  (\eic,I) & \mapsto & c(\eic,I).
\end{eqnarray}
Now we will show that this arrow is bijective. The proof that this bijection is additive, i.e. that $c$ is an isomorphism of groups, is left to the reader. 
From now on let $[\gic]=G=[d^G:G^{-1} \rightarrow G^0]\in \cK^{[-1,0]}(\bS)$. \\

Injectivity:  Let $(\eic,I)$ be an object of $\Psi_{\lic^.}(\gic)$ such that the morphism $c(\eic,I)$ that it defines in $\cD(\bS)$ is the zero morphism.
The hypothesis that $c(\eic,I)$ is zero in $\cD(\bS)$ implies that there exists a resolution of $G$ 
$$V^0 \longrightarrow V^1 \longrightarrow V^2 \longrightarrow ...$$
 and a quasi isomorphism
\begin{equation}
\xymatrix{
  0  \ar[r] & E\ar[d]_{v^0} \ar[r]^{j} & P \ar[d]^{v^1} \ar[r]  & 0 &\\
 0  \ar[r] & V^0 \ar[r]^{k} & V^1 \ar[r] & V^2 \ar[r] & ...
}
\label{eq:qiso}
\end{equation}
such that the composite 
\[
  \xymatrix{
  R \ar[r]^{D^R} & Q\ar[d]_{i} \ar[r]^{D^Q} & P \ar[d]^{id_P} \ar[r]  & 0 &\\
  0  \ar[r] & E\ar[d]_{v^0} \ar[r]^{j} & P \ar[d]^{v^1} \ar[r]  & 0 &\\
 0  \ar[r] & V^0 \ar[r]^{k} & V^1 \ar[r] & V^2 \ar[r] & ...
}
\]
is homotopic to zero. We can assume $V^i \in \cK^{[-1,0]}(\bS)$ for all $i$ and $V^i=0$ for $i \geq 2$ (instead of the complex of complexes $(V^i)_i$ consider its good truncation in degree 1). Since the complex of complexes $(V^i)_i$ is a resolution of $G$,
the short sequence of complexes 
$$ 0 \longrightarrow G \longrightarrow V^0 \longrightarrow V^1 \longrightarrow 0$$
is exact, i.e. $V^0$ is an extension of $W$ by $G$ (see Definition \ref{def:extcomplexes}).
Since the quasi-isomorphism (\ref{eq:qiso}) induces the identity on $G$, the extension $E$ is the fibred product $P \times_{V^1} V^0$ of $P$ and $V^0$ over $V^1$. Therefore, the morphism $s: P \rightarrow V^0$ inducing the homotopy $(v^0, v^1) \circ c(\eic,I) \sim 0$, i.e. satisfying $ k\circ s= v^1 \circ id_{P},$ factorizes through a morphism
\[h: P \longrightarrow E = P \times_{V^1} V^0\]
satisfying 
\[ j \circ h =id_{P} \qquad \qquad h \circ D^Q = i .\]
These two equalities mean that $st(h)$ splits the extension $\eic$, which is therefore the trivial extension of $\pic$ by $\gic$, and that $st(h)$ is compatible with the trivializations $I$.
Hence we can conclude that the object $(\eic,I)$ lies in the equivalence class of the zero object of $\Psi_{\lic^.}(\gic)$.\\

Surjectivity: Now we show that for any morphism $f$ of ${\Hom}_{\cD(\bS)}({\Tot}([\lic^.]),G[1])$, there is an element of $\Psi_{\lic^.}^1(\gic)$ whose image via $c$ is $f$.
The hypothesis that $f$ is an element of $\cD(\bS)$ implies that there exists a resolution of $G$ 
$$V^0 \longrightarrow V^1 \longrightarrow V^2 \longrightarrow ...$$
such that the morphism $f$ can be described in the category ${\cH}(\bS)$ via the following diagram
	\begin{equation}\label{eq:su1}
  \xymatrix{
  R \ar[r]^{D^R} & Q\ar[d]_{v^0} \ar[r]^{D^Q} & P \ar[d]^{v^1} \ar[r]  & 0 &\\
 0  \ar[r] & V^0 \ar[r]^{k} & V^1 \ar[r] & V^2 \ar[r] &...
}
\end{equation}
We can assume $V^i \in \cK^{[-1,0]}(\bS)$ for all $i$ and $V^i=0$ for $i \geq 2$ (instead of the complex of complexes $(V^i)_i$ consider its good truncation in degree 1). Since the complex of complexes $(V^i)_i$ is a resolution of $G$,
the short sequence of complexes 
$$ 0 \longrightarrow G \longrightarrow V^0 \longrightarrow V^1 \longrightarrow 0$$
is exact, i.e. $V^0$ is an extension of $V^1$ by $G$ (see Definition \ref{def:extcomplexes}).
 Consider the extension of $P$ by $G$
 \[Z=(v^1)^*V^0 = V^0 \times_{V^1} P\]
obtained as pull-back of $V^0$ via $w:P \rightarrow V^1.$ The condition $v^1 \circ D^Q = k \circ v^0 $ implies that $v^0 : Q\rightarrow V^0$ factories through a morphism
\[z:  Q \rightarrow Z\]
satisfying $l \circ z = D^Q$, with $l: Z \rightarrow P$ the canonical surjection of the extension $Z$. Moreover the conditions that $ v^0 \circ D^R $ and $D^Q \circ D^R $ are homotopic to zero furnish that also $z \circ D^R$ is homotopic to zero. 
Therefore the datum $(st(Z),st(z))$ is an object of the category  $\Psi_{\lic^.}(\gic)$. Consider now the morphism $c(st(Z),st(z)): {\Tot}([\lic^.]) \rightarrow G[1]$ associated to
$(st(Z),st(z))$. By construction, the morphism $f$ (\ref{eq:su1})
is the composite of the morphism $c(st(Z),st(z))$ 
\[
\xymatrix{
 R  \ar[d] \ar[r]^{D^R} & Q \ar[d]_{z} \ar[r]^{D^Q} & P \ar[d]^{id_P} \ar[r]  & 0 \\
0 \ar[r] & Z  \ar[r]^{l} & P  \ar[r] & 0
}
\]
with the morphism 
\[\xymatrix{
  0  \ar[r] & Z\ar[d]_{h} \ar[r]^{l} & P \ar[d]^{v^1} \ar[r]  & 0 \\
 0  \ar[r] & V^0 \ar[r]^{k} & V^1 \ar[r] & 0  ,
} \]
where $h: Z=(v^1)^*V^0 \rightarrow V^0$ is the canonical projection underlying the pull-back $Z$. Since this last morphism is a morphism of resolutions of $G$ (inducing the identity on $G$), we can conclude that in the derived category $\cD(\bS)$ the morphism $f : {\Tot}([\lic^.]) \rightarrow G[1]$ (\ref{eq:su1}) is the morphism $c(st(Z),st(z))$.\\

Using the above homological description of the groups $\Psi_{\lic^.}^i(\gic)$ for $i=-1,0,1$ we can study how the 2-category $\Psi_{\lic^.}(\gic) $ varies 
with respect to the complex $\lic^.$. Consider another complex $\lic'^.: \ric' \rightarrow \qic' \rightarrow \pic' \rightarrow 0$ and a morphism of complexes
\[ F^.: \lic'^. \longrightarrow \lic^. \]
given by the following commutative diagram (modulo isomorphisms of additive functors)
\begin{equation}\label{variance:1}
\xymatrix{
 \ric' \ar[d]_{F^{-2}} \ar[r]^{D^{\ric'}} & \qic'  \ar[d]_{F^{-1}} \ar[r]^{D^{\qic'}} & \pic'  \ar[d]^{F^0} \ar[r]  & 0 \\
 \ric  \ar[r]_{D^{\ric}} & \qic  \ar[r]_{D^{\qic}} & \pic  \ar[r] & 0.
}
\end{equation}
The morphism $F^.$ defines a canonical 2-functor
\[
(F^.)^*: \Psi_{\lic^.}(\gic) \longrightarrow \Psi_{\lic'^.}(\gic)
\]
as follows: if $(\eic,I)$ is an object of $\Psi_{\lic^.}(\gic)$, $(F^.)^*(\eic,I)$ is the object $(\eic',I')$ where
\begin{itemize}
  \item $\eic'$ is the extension $(F^0)^*\eic$ of $\pic'$ by $\gic$
obtained as pull-back of $\eic$ via $F^0: \pic' \rightarrow \pic$;
  \item $I'$ is the trivialization $(F^{-1})^*I $ of $(D^{\qic'})^*\eic'$ induced by the trivialization $I$ of $(D^{\qic})^*\eic$
via the commutativity of the first square of~(\ref{variance:1}).
\end{itemize}
The commutativity of the diagram~(\ref{variance:1}) implies that $(\eic',I')$ is in fact an object of $\Psi_{\lic'^.}(\gic)$ (the condition $I' \circ D^{\qic'} \cong 0$ is easily deducible from the corresponding conditions on $I$ and from the commutativity of the diagram~(\ref{variance:1})).
 
\begin{prop}\label{prop:psi-equiv}
Let $F^.: \lic'^. \rightarrow \lic^.$ be morphism of complexes.
The corresponding 2-functor $(F^.)^*: \Psi_{\lic^.}(\gic) \rightarrow \Psi_{\lic'^.}(\gic)$ is an equivalence of 2-categories if and only if
the homomorphisms 
$${\h}^i\big({\Tot}(F^.)\big):{\h}^i\big({\Tot}([\lic'^.])\big) \longrightarrow {\h}^i\big({\Tot}([\lic^.])\big)  \qquad i=-1,0,1$$
 are isomorphisms.
\end{prop}

\begin{proof} The 2-functor $(F^.)^*: \Psi_{\lic^.}(\gic) \rightarrow \Psi_{\lic'^.}(\gic)$ defines the following homomorphisms
\begin{equation}\label{1}
    ((F^.)^*)^i: \Psi_{\lic^.}^i(\gic) \longrightarrow \Psi_{\lic'^.}^i(\gic) \qquad i=-1,0,1.
\end{equation}
On the other hand the morphism of complexes  $F^.: \lic'^. \rightarrow \lic^.$ defines the following homomorphisms
\begin{equation}\label{2}
    ({\Tot}(F^.))^i: {\Ext}^i\big({\Tot}([\lic^.]), -\big) \longrightarrow {\Ext}^i\big({\Tot}([\lic'^.]), -\big) \qquad i \in {\ZZ}.
\end{equation}
Since the homomorphisms~(\ref{1}) and~(\ref{2}) are compatible with the canonical isomorphisms obtained in Theorem~\ref{thm:psi-ext},
 the following diagrams (with $i=-1,0,1$) are commutative:
\[
\begin{array}{ccc}
 \Psi_{\lic^.}^i(\gic)&\rightarrow &{\Ext}^i\big({\Tot}([\lic^.]),[\gic]\big)\\
 \downarrow &  & \downarrow \\
 \Psi_{\lic'^.}^i(\gic)& \rightarrow & {\Ext}^i\big({\Tot}([\lic'^.]),[\gic]\big).
\end{array}
\]
The 2-functor $(F^.)^*: \Psi_{\lic^.}(\gic) \rightarrow \Psi_{\lic'^.}(\gic)$ is an equivalence of 2-categories if and only if
 the homomorphisms~(\ref{1}) are isomorphisms, and so using the above commutative diagrams 
we are reduced to prove that the homomorphisms~(\ref{2}) are isomorphisms if and only if the homomorphisms
 ${\h}^i\big({\Tot}(F^.)\big):{\h}^i\big({\Tot}([\lic'^.])\big) \rightarrow {\h}^i\big({\Tot}([\lic^.])\big)$ 
are isomorphisms. This last assertion is clearly true.
\end{proof}

\section{Geometrical description of $\Psi_{\lic^.}(\gic)$}

In this section we switch from cohomological notation to homological.

Let $\pic$ be a strictly commutative Picard $\bS$-stack. Because of the new homological notations the complex 
 $[\pic]=P=[d_P:P_1 \rightarrow P_0]$ has $P_1$ in degree 1 and $P_0$ in degree 0. We start constructing \textbf{a canonical flat partial resolution} for the complex $[\pic]$. We introduce the following notations: if $A$ is an abelian sheaf on $\bS$, we denote by $[a]$ the element of ${\ZZ}[A](U)$ defined by the point $a$ of $A(U)$ with $U$ a object of $\bS$. In an analogous way, if $a,b$ and $c$ are points of $A(U)$ we denote by $[a,b]$, $[a,b,c]$ the elements of ${\ZZ}[A\times A](U)$ and ${\ZZ}[A\times A \times A](U)$ respectively. Denote by  
 ${\ZZ}[P]=[{\ZZ}[d_P]: {\ZZ}[P_1] \rightarrow {\ZZ}[P_0]]$ the complex of abelian sheaves generated by $P$, 
where ${\ZZ}[P_i]$ is the abelian sheaf generated by $P_i$ for $i=1,0$ (see \cite{SGA4} Expos\'e IV 11).
Moreover let ${\ZZ}[\pic]$ the strictly commutative Picard $\bS$-stack $st({\ZZ}[P])$ corresponding to the complex ${\ZZ}[P]$ via (\ref{st}).

Consider the following complexes of strictly commutative Picard $\bS$-stacks
$$\lic.(\pic) : \qquad  {\ZZ}[\pic \times \pic] + {\ZZ}[\pic \times \pic \times \pic] \stackrel{D_1}{\longrightarrow} {\ZZ}[\pic \times \pic] \stackrel{D_0}{\longrightarrow} {\ZZ}[\pic] \longrightarrow 0$$ 
with $\lic_0(\pic)={\ZZ}[\pic], \lic_1(\pic)={\ZZ}[\pic \times \pic]$ and  $\lic_2(\pic)= {\ZZ}[\pic \times \pic] + {\ZZ}[\pic \times \pic \times \pic]$ in degrees
0,1 and 2 respectively. The differential operators are defined as follows: if $ p_1, p_2,p_3 \in {\ZZ}[\pic],$ we set 
\begin{eqnarray}
\nonumber D_0[p_1,p_2] &=& [p_1+p_2] -[p_1]-[p_2]\\
\label{Dd} D_1[p_1,p_2] &=& [p_1,p_2] -[p_2,p_1]\\
\nonumber D_1[p_1,p_2,p_3] &=& [p_1+p_2,p_3] -[p_1,p_2+p_3]+[p_1,p_2]-[p_2,p_3].
\end{eqnarray}
Consider also the additive functor $\epsilon : {\ZZ}[\pic] \rightarrow \pic$ defined by $\epsilon([p])=p$ for any $p \in \pic.$
This additive functor is an augmentation map for the complex $\lic.(\pic).$
 Note that the relation $\epsilon \circ D_0 =0$ is just the group law $+: \pic \times \pic \rightarrow \pic$ on $\pic$, and the relation $D_0\circ D_1 =0$ decomposes in two relations which express the commutativity $\tau$ (\ref{com}) and the associativity $\sigma$ (\ref{ass}) of the group law on $\pic$. This augmented complex $\lic.(\pic)$ depends functorially on $\pic$: in fact, any additive functor $F:\pic \rightarrow \pic'$ furnishes a commutative diagram
\[\begin{array}{ccc}
 \lic.(\pic) & \stackrel{\lic.(F)}{\longrightarrow} & \lic.(\pic') \\
 \epsilon  \downarrow &  & \downarrow \epsilon \\
  \pic & \stackrel{F}{\longrightarrow} & \pic'.
\end{array}\]
Moreover the components of the complex $\lic.(\pic)$ are flat since they are free $\ZZ$-modules. In order to conclude that $\lic.(\pic)$ is a canonical flat partial resolution of $\pic$ we need the following Lemma. Let $\gic$ be a strictly commutative Picard $\bS$-stack.

\begin{lem}\label{lem:ext-geom}
The 2-category ${\cExt}(\pic,\gic)$ of extensions of $\pic$ by $\gic$ is equivalent to the 2-category $\Psi_{\lic.(\pic)}(\gic):$
 $$ {\cExt}(\pic,\gic) \cong \Psi_{\lic.(\pic)}(\gic).$$
\end{lem}

\begin{proof} In order to describe explicitly the objects of the category $\Psi_{\lic.(\pic)}(\gic)$ we use the description~(\ref{rem:Z[I]exttor}) in terms of torsors, of the extensions of complexes whose entries are free commutative groups:
\begin{itemize}
  \item an extension of ${\ZZ}[\pic]$ by $\gic$ is a $(\gic)_{\pic}$-torsor,
  \item  an extension of
$ {\ZZ}[ \pic \times \pic] $ by $\gic$ is a $(\gic)_{\pic \times \pic}$-torsor, and finally
  \item an extension of ${\ZZ}[\pic \times \pic]+ {\ZZ}[\pic \times \pic \times \pic] $ by $\gic$ consists of a couple of a $(\gic)_{\pic \times\pic}$-torsor
      and a $(\gic)_{\pic \times \pic \times \pic}$-torsor.
\end{itemize}
According to these considerations an object $(\eic,I)$ of $\Psi_{\lic.(\pic)}(\gic)$ consists of

(1) an extension $\eic$ of ${\ZZ}[\pic] $ 
by $\gic$, i.e. a $\gic$-torsor $\eic$ over $\pic$. Since ${\Ext}^1({\ZZ}[\mathbf{1}],\gic)=0$, it exists
a trivialization $T$ of the pull-back $\mathbf{1}^*\eic$ of the $\gic$-torsor $\eic$ via the additive functor $\mathbf{1}: \mathbf{1} \rightarrow \pic$;
  
(2) a trivialization $I$ of the extension $D_0^*\eic$ of $ {\ZZ}[ \pic \times \pic] $ by $\gic$  obtained as pull-back of $\eic$ via $D_0: {\ZZ}[\pic \times \pic] \rightarrow {\ZZ}[\pic]$, i.e. a trivialization $I$ of the $\gic$-torsor $D_0^*\eic$ over $\pic \times \pic$ obtained as pull-back of $\eic$ via $ D_0$. This trivialization can be interpreted as a morphism of $\gic$-torsors $\eic$:
      \[
      M: p_1^* \; \eic \wedge p_2^* \; \eic \longrightarrow +^* \; \eic
      \]
where $p_i: \pic \times \pic  \rightarrow \pic$ are the projections  and $ +: \pic \times \pic  \rightarrow \pic $ is the group law of $ \pic$. The restriction of $M$ over $\mathbf{1} \times \mathbf{1}$ is compatible with the trivialization $T$.\\
The compatibility of $I$ with the relation $D_0 \circ D_1 =0$ imposes on the datum $(\eic,T,M)$ two relations through the two torsors over $\pic \times \pic$ and $\pic \times \pic \times \pic$. These two relations are the isomorphism $\alpha$ of morphisms of $\gic$-torsors over $\pic \times \pic \times \pic$ described in (\ref{eq:alpha}) and the isomorphism $\chi$ of morphisms of $\gic$-torsors over $\pic \times \pic $ described in (\ref{eq:chi}), which satisfy the equalities (\ref{ext-tor:1}) and (\ref{ext-tor:2}). Moreover, the restriction of $\alpha$ over $\mathbf{1} \times \mathbf{1} \times \mathbf{1} $ is the identity and since we are dealing with extensions of strictly commutative Picard stacks, the pull-back $D^* \chi$ of $\chi$ via the diagonal $D:\pic \rightarrow \pic \times \pic$ is the identity and the composite of $\chi$ with itself is the identity.

Hence by Theorem \ref{thm:ext-tor} the object $(\eic,T,M,\alpha,\chi)$ of $\Psi_{\lic.(\pic)}(\gic)$ is an extension of $\pic$ by $\gic$ and we can conclude that the 2-category $\Psi_{\lic.(\pic)}(\gic)$ is equivalent to the 2-category ${\cExt}(\pic,\pic')$. 
\end{proof}

\begin{prop}\label{resolpart}
The augmentation map $\epsilon  : \lic.(\pic) \rightarrow \pic$ induces the isomorphisms ${\h}_i({\Tot}( \lic.(\pic))) \cong {\h}_i([\pic])$ for $i=1,0,-1$.
\end{prop}

\begin{proof} Applying Proposition~\ref{prop:psi-equiv} to
 the augmentation map $\epsilon  : \lic.(\pic) \rightarrow \pic$,
 we just have to prove that for any strictly commutative Picard $\bS$-stack $\gic$
 the 2-functor
 $$\epsilon ^* : \Psi_{\pic}(\gic) \rightarrow
 \Psi_{\lic.(\pic)}(\gic)$$
 is an equivalence of 2-categories (in the symbol $\Psi_{\pic}(\gic)$, $\pic$ is seen as a complex whose only non trivial entry is $\pic$ in degree 0). According to Definition~\ref{psi}, it is clear that the 2-category
 $\Psi_{\pic}(\gic)$ is the 2-category ${\cExt}(\pic,\gic)$
 of extensions of $\pic$ by $\gic$. On the other hand, by Lemma~\ref{lem:ext-geom} also the 2-category
 $\Psi_{\lic.(\pic)}(\gic)$ is equivalent to the 2-category ${\cExt}(\pic,\gic)$. Hence we can conclude.
\end{proof}

Let $\pic,\qic$ and $\gic$ be three strictly commutative Picard $\bS$-stacks and let $\lic.(\pic), \lic.(\qic)$ be the canonical flat partial resolutions of $\pic$ and $\qic$ respectively. Denote by
 $\lic.(\pic, \qic)$ the complex $\lic.(\pic) \otimes \lic.(\qic).$

\begin{thm}\label{thm:psi-geom}
The 2-category ${\cBiext}(\pic,\qic;\gic)$ of biextensions of $(\pic,\qic)$ by $\gic$ is equivalent to the 2-category $\Psi_{\lic.(\pic,\qic)}(\gic):$
\[
    {\cBiext}(\pic,\qic;\gic) \cong \Psi_{\lic.(\pic,\qic)}(\gic)
\]
\end{thm}

\begin{proof}
Explicitly, the non trivial components of $\lic.(\pic,\qic)$ are
\begin{eqnarray}
\nonumber \lic_0(\pic,\qic) &=& \lic_0(\pic) \otimes \lic_0(\qic)\\
\nonumber  &=& {\ZZ}[\pic \times \qic] \\
\nonumber \lic_1(\pic,\qic) &=& \lic_0(\pic) \otimes \lic_1(\qic)+\lic_1(\pic) \otimes \lic_0(\qic)\\
\nonumber  &=&{\ZZ}[\pic \times \qic \times \qic]+ {\ZZ}[\pic \times \pic\times \qic] \\
\nonumber \lic_2(\pic,\qic) &=&\lic_0(\pic) \otimes \lic_2(\qic)+\lic_2(\pic) \otimes \lic_0(\qic)+\lic_1(\pic) \otimes \lic_1(\qic)\\
 \nonumber  &=&{\ZZ}[\pic \times \qic \times \qic]+{\ZZ}[\pic \times \qic \times \qic \times \qic]+ \\
\nonumber  && {\ZZ}[\pic \times \pic \times \qic]+{\ZZ}[\pic \times \pic \times \pic \times \qic]+\\
\nonumber  && {\ZZ}[\pic \times \pic \times \qic \times \qic]
\end{eqnarray}
The differential operators of the complex $\lic.(\pic,\qic)$ have to satisfy the following conditions: the sequences
  \begin{equation}\label{exact1}
   {\ZZ}[\pic \times \qic \times \qic]+{\ZZ}[\pic \times \qic \times \qic \times \qic]
   \stackrel{id_{\pic} \times D^\qic_1 }{\longrightarrow}
   {\ZZ}[\pic \times \qic \times \qic]
   \stackrel{id_{\pic} \times D^\qic_0}{\longrightarrow}
   {\ZZ}[\pic \times \qic]
  \end{equation}
  \begin{equation}\label{exact2}
    {\ZZ}[\pic \times \pic \times \qic]+{\ZZ}[\pic \times \pic \times \pic \times \qic]
   \stackrel{D^\pic_1 \times id_{ \qic}}{\longrightarrow}
   {\ZZ}[\pic \times \pic \times \qic]
   \stackrel{ D^\pic_0 \times id_{\qic} }{\longrightarrow}
   {\ZZ}[\pic \times \qic]
  \end{equation}
 are exact and the diagram  
 \begin{equation}\label{anti1}
      \begin{array}{ccc}
        {\ZZ}[\pic \times \pic \times \qic \times \qic] & \stackrel{id_{\pic \times \pic} \times D^\qic_0 }{\longrightarrow}  & {\ZZ}[\pic \times \pic \times \qic]  \\
 \scriptstyle{D^\pic_0 \times id_{\qic \times \qic}}    \downarrow & & \downarrow \scriptstyle{D^\pic_0 \times id_{\qic}}\\
        {\ZZ}[\pic \times \qic \times \qic]  &  \stackrel{ id_{\pic} \times D^\qic_0 }{\longrightarrow}  & {\ZZ}[\pic \times \qic] \\
      \end{array}
 \end{equation}
 is anticommutative.

In order to describe explicitly the objects of $\Psi_{\lic.(\pic,\qic)}(\gic)$ we use the description~(\ref{rem:Z[I]exttor}) in terms of torsors, of the extensions of complexes whose entries are free commutative groups:
\begin{itemize}
  \item an extension of $\lic_0(\pic,\qic)$ by $\gic$ is a $(\gic)_{\pic \times \qic}$-torsor,
  \item  an extension of
$\lic_1(\pic,\qic))$ by $\gic$ consists of a $(\gic)_{\pic \times \qic \times \qic}$-torsor and a 

$(\gic)_{\pic \times \pic \times \qic}$-torsor,
  \item an extension of $\lic_2(\pic,\qic))$ by $\gic$ consists of a system of 5 torsors under the groups deduced from $\gic$ by base change over the bases $\pic \times \qic \times \qic,~\pic \times \qic \times \qic \times \qic,~ \pic \times \pic \times \qic,~
\pic \times \pic \times \pic \times \qic,~\pic \times \pic \times \qic \times \qic.$
\end{itemize}
By these considerations an object $(\eic,I)$ of $\Psi_{\lic.(\pic,\qic)}(\gic)$ consists of

(1) an extension $\eic$ of ${\ZZ}[\pic \times \qic] $ by $\gic$, i.e. 
   a $\gic$-torsor $\eic$ over $\pic \times \qic$. Since ${\Ext}^1({\ZZ}[\mathbf{1}\times \mathbf{1}],\gic)=0$, it exists
 two trivializations $T^\pic$ and $T^\qic$ of the pull-back $(\mathbf{1}\times \mathbf{1})^*\eic$ of the $\gic$-torsor $\eic$ via the additive functor $ \mathbf{1} \times \mathbf{1} \rightarrow \pic \times \qic$;
  
 (2) a trivialization $I$ of the extension $( id_\pic \times D_0^\qic + D_0^\pic \times id_\qic )^*\eic $ of 
  ${\ZZ}[\pic \times \qic \times \qic]+ {\ZZ}[\pic \times \pic \times \qic]$ by $\gic$ obtained as pull-back of $\eic$ via 
$$(id_\pic \times D_0^\qic + D_0^\pic \times id_\qic ): {\ZZ}[\pic \times \qic \times \qic]+  {\ZZ}[\pic \times \pic \times \qic]
\longrightarrow  {\ZZ}[\pic \times \qic ],$$
 i.e. a couple of trivializations of the couple of $\gic$-torsors over $\pic \times \qic \times \qic$ and $\pic \times \pic \times \qic$ which are the pull-back of $\eic$ via $(id_\pic \times D_0^\qic + D_0^\pic \times id_\qic )$. 
 These trivializations can be interpreted as a morphism of $\gic_\pic$-torsors over $\qic_\pic \times \qic_\pic $ 
and a morphism of $\gic_\qic$-torsors over $\pic_\qic \times \pic_\qic$
 $$M^\qic: p_1^{\qic \;*} \; \bic_\pic \wedge p_2^{\qic \;*} \; \bic_\pic \longrightarrow +^{\qic \;*} \; \bic_\pic ,\quad 
 \; \; M^\pic: p_1^{\pic \;*} \; \bic_\qic \wedge p_2^{\pic \;*} \; \bic_\qic \longrightarrow +^{\pic \;*} \; \bic_\qic
$$ 
where $p_i^{\qic}: \qic_\pic \times \qic_\pic \rightarrow \qic_\pic $, $p_i^{\pic}: \pic_\qic \times \pic_\qic \rightarrow \pic_\qic $ 
are the projections ($i=1,2$) and $+^{\qic}:\qic_\pic \times \qic_\pic  \rightarrow \qic_\pic $, $+^{\pic}:\pic_\qic \times \pic_\qic  \rightarrow \pic_\qic $
are the group laws of $ \qic_\pic $ and of $ \pic_\qic $ respectively. Remark that the restriction of $M^\pic$ over $\mathbf{1} \times \mathbf{1}$ is compatible with the trivialization $T^\pic$ (idem for $M^\qic$).\\
Finally, the compatibility of $I$ with the relation 
 $$\big( id_\pic \times D_0^\qic + D_0^\pic \times id_\qic \big) \circ \big(id_\pic \times D_1^\qic + D_1^\pic \times id_\qic   +(D_0^\pic \times id_{\qic \times \qic}, id_{\pic \times \pic} \times D_0^\qic ) \big) =0$$
 imposes on the datum $(\eic,T^\pic,T^\qic,M^\pic,M^\qic)$ 5 relations of compatibility
       through the system of 5 torsors over $\pic \times \qic \times \qic,~\pic \times \qic \times \qic \times \qic,~ \pic \times \pic \times \qic,~\pic \times \pic \times \pic \times \qic,~\pic \times \pic \times \qic \times \qic$ arising from $\lic_2(\pic,\qic):$ 
   \begin{itemize}
         \item the exact sequence~(\ref{exact1}) furnishes two relations through the two torsors over $\pic \times \qic \times \qic$ and $\pic \times \qic \times \qic \times \qic$. These two relations are the isomorphism $\alpha^\qic$ of morphisms of $\gic$-torsors over $\pic \times \qic \times \qic \times \qic$ described in (\ref{eq:alpha}) and the isomorphism $\chi^\qic$ of morphisms of $\gic$-torsors over $\pic \times \qic \times \qic $ described in (\ref{eq:chi}), which satisfy the equalities (\ref{ext-tor:1}) and (\ref{ext-tor:2}). Moreover, the restriction of $\alpha^\qic$ over $\mathbf{1} \times \mathbf{1} \times \mathbf{1} $ is the identity and since we are dealing with extensions of strictly commutative Picard stacks, the pull-back $D^* \chi^\qic$ of $\chi^\qic$ via the diagonal morphism is the identity and the composite of $\chi^\qic$ with itself is the identity. Hence by Theorem \ref{thm:ext-tor} the $\gic$-torsor $\eic$ is endowed with a structure of extension of $(\qic)_{\pic}$ by $(\gic)_{\pic}$;
          \item the exact sequence~(\ref{exact2}) expresses two relations through the two torsors over $\pic \times \pic \times \qic$ and $\pic \times \pic \times \pic \times \qic$. These two relations are the isomorphism $\alpha^\pic$ of morphisms of $\gic$-torsors over $\pic \times \pic \times \pic \times \qic$ described in (\ref{eq:alpha}) and the isomorphism $\chi^\pic$ of morphisms of $\gic$-torsors over $\pic \times \pic \times \qic $ described in (\ref{eq:chi}), which satisfy the equalities (\ref{ext-tor:1}) and (\ref{ext-tor:2}). Moreover, the restriction of $\alpha^\pic$ over $\mathbf{1} \times \mathbf{1} \times \mathbf{1} $ is the identity and since we are dealing with extensions of strictly commutative Picard stacks, the pull-back $D^* \chi^\pic$ of $\chi^\pic$ via the diagonal morphism is the identity and the composite of $\chi^\pic$ with itself is the identity. Hence by Theorem \ref{thm:ext-tor} the $\gic$-torsor $\eic$ is endowed with a structure of 
           extension of $(\pic)_{\qic}$ by $(\gic)_{\qic}$;
          \item the anticommutative diagram~(\ref{anti1}) furnishes a relations through the torsor over $\pic \times \pic \times \qic \times \qic$. This relation is the isomorphism $\beta$ of morphisms of $\gic_{\pic \times \qic}$-torsors over $(\pic \times \qic) \times (\pic \times \qic)$ described in (\ref{eq:compbiextpic}). This means that the two structures 
           of extension of $(\qic)_{\pic}$ by $(\gic)_{\pic}$ and of  extension of $(\pic)_{\qic}$ by $(\gic)_{\qic}$ that we have on the $\gic$-torsor $\eic$ are compatible.
          \end{itemize}
The object $(\eic,T^\pic,T^\qic,M^\pic,M^\qic,\alpha^\pic,\alpha^\qic,\chi^\pic,\chi^\qic,\beta)$ of $\Psi_{\lic.(\pic,\qic)}(\gic)$ is therefore a biextension of $(\pic,\qic)$ by $\gic$.
We can then conclude that the 2-category $\Psi_{\lic.(\pic,\qic)}(\gic)$ is equivalent to the 2-category ${\cBiext}(\pic,\qic,\gic)$.
\end{proof}

\section{Proof of Theorem 0.1 (\textbf{a})}

Let $\pic,\qic$ and $\gic$ be three strictly commutative Picard $\bS$-stacks.

Denote respectively by $\lic.(\pic)$ and $\lic.(\qic)$ the canonical flat partial resolutions of $\pic$ and $\qic$ introduced in \S 9. According to Proposition~\ref{resolpart}, there exists arbitrary flat resolutions 
$\lic.'(\pic), \lic.'(\qic)$ of $\pic$ and $\qic$ such that we have the following isomorphisms for $j=-1,0,1$
$${\Tot}( \lic.(\pic))_j  \cong {\Tot}( \lic.'(\pic))_j \qquad {\Tot}( \lic.(\qic))_j  \cong {\Tot}( \lic.'(\qic))_j.$$
Hence it exists two canonical morphisms of complexes
$$ \lic.(\pic) \longrightarrow \lic.'(\pic)  \qquad
\lic.(\qic) \longrightarrow \lic.'(\qic)$$
inducing a canonical morphism between the corresponding total complexes
$${\Tot}([\lic.(\pic) \otimes \lic.(\qic)]) \longrightarrow {\Tot} ([\lic.'(\pic) \otimes \lic.'(\qic)]) $$
which is an isomorphism in degrees -1, 0 and 1.
Denote by
 $\lic.(\pic, \qic)$ (resp. $\lic.'(\pic, \qic)$) the complex $\lic.(\pic) \otimes \lic.(\qic)$
 (resp. $\lic.'(\pic) \otimes \lic.'(\qic)$).
 Remark that ${\Tot}([\lic.'(\pic, \qic)])$ represents $[\pic] {\buildrel {\scriptscriptstyle \LL} \over \otimes}[\qic]$ in the derived category $\cD(\bS)$:
 \[{\Tot}([\lic.'(\pic, \qic)]) = [\pic] {\buildrel {\scriptscriptstyle \LL} \over \otimes} [\qic].\]
 By Proposition~\ref{prop:psi-equiv} we have the equivalence of categories
$$\Psi_{\lic.(\pic,\qic)}(\gic) \cong \Psi_{\lic.'(\pic,\qic)}(\gic).$$
Hence applying Theorem~\ref{thm:psi-geom}, which furnishes the following geometrical description of the category $\Psi_{\lic.(\pic,\qic)}(\gic)$:
$$\Psi_{\lic.(\pic,\qic)}(\gic)  \cong {\bBiext}(\pic,\qic;\gic), $$
and applying Theorem~\ref{thm:psi-ext}, which furnishes the following homological description of the groups $\Psi_{\lic.'(\pic,\qic)}^i(\gic)$ for $i=-1,0,1$:
$$\Psi_{\lic.'(\pic,\qic)}^i(\gic)\cong  {\Ext}^i\big({\Tot}\big([\lic.'(\pic,\qic)]\big),[\gic]\big)\cong {\Ext}^i([\pic] {\buildrel {\scriptscriptstyle \LL} \over \otimes} [\qic],[\gic]),$$
 we get Theorem~\ref{intro:mainthm}, i.e.
${\Biext}^i(\pic,\qic;\gic) \cong {\Ext}^i([\pic]{\buildrel {\scriptscriptstyle \LL} \over \otimes} [\qic],[\gic])$ for $i=-1,0,1. $


\end{document}